%% file: example_paper.tex
\documentclass{article}

\usepackage{hyperref}       
\usepackage{url}            
\usepackage{booktabs}       
\usepackage{amsfonts}       
\usepackage{nicefrac}       
\usepackage{microtype}      
\usepackage{pifont}
\usepackage{amssymb}

\usepackage{algorithm}
\usepackage{algorithmic}
\usepackage{hyperref}
\usepackage{natbib}

\usepackage[a4paper,top=3cm,bottom=2cm,left=3cm,right=3cm,marginparwidth=1.75cm]{geometry}

\usepackage{amsmath,amssymb,amsthm}
\usepackage{pgfplots}
\usepackage{color}
\usepackage{flushend}
\usepackage{multirow}
\usepackage{rotating}
\usepackage{appendix}
\usepackage{tikz,pgfplots}
\usepackage{subcaption}

\input{mysymbol.sty}

\newtheorem{assumption}{\hspace{0pt}\bf Assumption}
\newtheorem{remark}{\hspace{0pt}\bf Remark}
\newtheorem{definition}{\hspace{0pt}\bf Definition}
\newtheorem{proposition}{\hspace{0pt}\bf Proposition}
\newtheorem{lemma}{\hspace{0pt}\bf Lemma}
\newtheorem{theorem}{\hspace{0pt}\bf Theorem}

\newcommand{\xmark}{\ding{55}}%

\begin{document}

\title{A Unified Analysis of Extra-gradient and Optimistic Gradient Methods for Saddle Point Problems: Proximal Point Approach}

\date{}

\author{Aryan Mokhtari\thanks{The authors are in alphabetical order.} \thanks{Laboratory for Information \& Decision Systems, Massachusetts Institute of Technology, Cambridge, MA, USA. aryanm@mit.edu}, Asuman Ozdaglar$^*$\thanks{Department of Electrical Engineering and Computer Science, Massachusetts Institute of Technology, Cambridge, MA, USA. asuman@mit.edu.}, Sarath Pattathil$^*$\thanks{Department of Electrical Engineering and Computer Science, Massachusetts Institute of Technology, Cambridge, MA, USA. sarathp@mit.edu.}
}

\maketitle
\begin{abstract}
\noindent In this paper we consider solving saddle point problems using two variants of Gradient Descent-Ascent algorithms, Extra-gradient (EG) and Optimistic Gradient Descent Ascent (OGDA) methods. We show that both of these algorithms admit a unified analysis as approximations of the classical proximal point method for solving saddle point problems. This viewpoint enables us to develop a new framework for analyzing EG and OGDA for bilinear and strongly convex-strongly concave settings. Moreover, we use the proximal point approximation interpretation to generalize the results for OGDA for a wide range of parameters. 
\end{abstract}

\input{introduction}

\input{preliminaries}
\input{proximal}

\input{ogd}
\input{extragradient}
\input{simulations}

\section{ Conclusions} \label{sec:conclusions}
We consider discrete time gradient based methods for solving convex-concave saddle point problems, with a focus on the Extra-gradient (EG) and the Optimistic Gradient Descent Ascent (OGDA) methods. We show that EG and OGDA can be seen as approximations of the classical Proximal Point (PP) method. We provide linear rate estimates for the bilinear and strongly convex-strongly concave saddle point problems for EG and OGDA as well as their generalizations.

%
%



\bibliography{references}
\bibliographystyle{icml2019}

\input{appendix}

\end{document}

%% file: introduction.tex

\section{Introduction}\label{sec:intro}

In this paper, we study the following saddle point problem
\begin{equation}\label{main_prob}
\min_{\bbx \in \reals^m}\max_{\bby \in \reals^n} \ f(\bbx,\bby),
\end{equation}
where the function $f:\reals^m\times \reals^n \to \reals$ is a convex-concave function, i.e., $f(\cdot, \bby)$ is convex for all $\bby \in \reals^n$ and $f(\bbx, \cdot)$ is concave for all $\bbx \in \reals^m$. We are interested in computing a saddle point of problem \eqref{main_prob} defined as a pair $(\bbx^*, \bby^*) \in \reals^m \times \reals^n $ that satisfies the condition 
$$ f(\bbx^*, \bby) \leq f(\bbx^*, \bby^*) \leq f(\bbx, \bby^*), $$
for all $\bbx \in \reals^m, \bby \in \reals^n$. This problem formulation appears in several areas, including zero-sum games \citep{basar1999dynamic}, robust optimization \citep{ben2009robust}, robust control \citep{hast2013pid}  and more recently in machine learning in the context of Generative Adversarial Networks (GANs); see \citep{goodfellow2014generative} for an introduction to GANs and \citep{pmlr-v70-arjovsky17a} for the formulation of Wasserstein GANs.

Motivated by the interest in computational methods for solving the minmax problem in \eqref{main_prob}, in this paper we consider convergence rate analysis of discrete-time gradient based optimization algorithms for finding a saddle point of problem \eqref{main_prob}. We focus on Extra-gradient (EG) and Optimistic Gradient Descent Ascent (OGDA) methods, which have attracted much attention in the recent literature because of their superior empirical performance in GAN training (see \cite{DBLP:journals/corr/abs-1802-06132}, \cite{DBLP:journals/corr/abs-1711-00141}). EG is a classical method which was introduced in \cite{korpelevich1976extragradient}. Its linear rate of convergence for smooth and strongly convex-strongly concave functions $f(\bbx, \bby)$ and bilinear functions, i.e., $f(\bbx, \bby) = \bbx^{\top} \bbA \bby$, was established in the variational inequality literature (see \citep{facchinei2007finite} and \citep{tseng_1}). The convergence properties of OGDA were recently studied in \cite{DBLP:journals/corr/abs-1711-00141}, which showed the convergence of the iterates to a neighborhood of the solution when the objective function is bilinear.  The recent paper \cite{DBLP:journals/corr/abs-1802-06132} used a dynamical system approach to prove the linear convergence of the OGDA and EG methods for the special case when $f(\bbx,\bby) = \bbx^{\top} \bbA \bby$ and the matrix $\bbA$ is square and full rank. It also presented a linear convergence rate of the vanilla Gradient Ascent Descent (GDA) method when the objective function $f(\bbx, \bby)$ is strongly convex-strongly concave. In a recent paper \cite{gidel2018variational}, a variant of the EG method is considered, relating it to OGDA updates, and show the linear convergence of the corresponding EG iterates in the case where $f(\bbx, \bby)$ is strongly convex-strongly concave\footnote{$f(\bbx,\bby)$ is strongly convex-strongly concave when it is strongly convex with respect to $\bbx$ and strongly concave with respect to $\bby.$} (though without showing the convergence rate for the OGDA iterates).

The previous works use disparate approaches to analyze EG and OGDA methods, obtaining results in several different settings and making it difficult to see the connections and unifying principles between these iterative methods. In this paper, we show that the update of EG and OGDA can be interpreted as approximations of the Proximal Point (PP) method, introduced in \cite{martinet1970breve} and studied in \cite{rockafellar1976monotone}. This viewpoint allows us to understand why EG and OGDA are convergent for a bilinear problem. It also enables us to generalize OGDA (in terms of parameters) and obtain new convergence rate results for these generalized algorithms for the bilinear case. Our results recover the linear convergence rate results of \cite{tseng_1} for EG and the linear rate results of \cite{DBLP:journals/corr/abs-1802-06132} for the bilinear case of OGDA. We obtain new linear convergence rate estimates for OGDA for the strongly convex-strongly concave case as well as linear convergence rates for the generalized OGDA method. 
\begin{figure}[t!]
  \centering
\includegraphics[width=0.5\columnwidth]{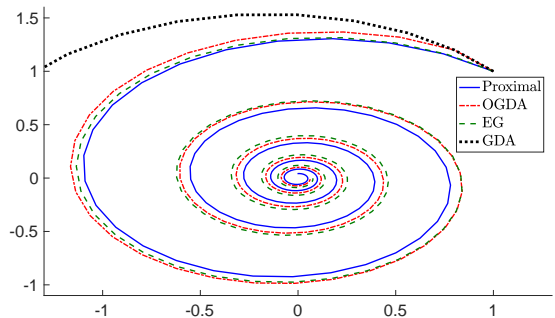}
\vspace{-2mm}
    \caption{Convergence trajectories of proximal point (PP), extra-gradient (EG), optimistic gradient descent ascent (OGDA), and gradient descent ascent (GDA) for $\min_x \max_y xy$. The proximal point method has the fastest convergence. EG and OGDA approximate the trajectory of PP and both converge to the optimal solution. The GDA method is the only method that diverges.
}\vspace{-2mm}
 \label{fig_trajec}
 \end{figure}



\textbf{Related Work.} The result in \cite{tseng_1} showed convergence of EG method to an $\epsilon$ optimal solution with iteration complexity of $\mathcal{O}(\kappa \log(1/\epsilon))$ (see Assumption \ref{ass:bili} and  Remark \ref{cond_numb_remark} for the definition of $\kappa$) , when the function $f(\bbx, \bby)$ is smooth and strongly convex-strongly concave and when $f$ is bilinear. A variational inequality perspective of saddle point problems was used in proving these results. More recently, \cite{DBLP:journals/corr/abs-1802-06132} analyzed EG and OGDA for the case when $f$ is bilinear, using a dynamical system perspective. The authors showed a complexity of $\mathcal{O}(\kappa \log(1/\epsilon))$ for OGDA and a complexity of $\mathcal{O}(\kappa^2 \log(1/\epsilon))$ for EG, without anlyzing the general strongly convex-strongly concave setting. In another recent independent work, \cite{gidel2018variational} analyzed the convergence of the OGDA method using the interpretation that OGDA is a variant of EG using extrapolation from the past.  In this connection, the OGDA iterates are the ``midpoints" 
whereas \citep{gidel2018variational} provides a convergence of the original points (not the OGDA iterates) to an error of $\eps$ in $\mathcal{O}(\kappa \log(1/\eps))$. In this paper, we establish an overall complexity of $\mathcal{O}(\kappa \log(1/\epsilon))$ for both OGDA and EG in bilinear and strongly convex-strongly concave settings by interpreting these methods as approximations of the proximal point method. The results of our paper are compared with existing results in Table~\ref{table:1}. Apart from the algorithms summarized in Table~\ref{table:1}, we also propose a generalized version of OGDA which extends the classical OGDA algorithm to a wider range of stepsize parameters and show its convergence for bilinear case.



There are several papers that study the convergence rate of algorithms for solving saddle point problems \textit{over a compact set}. Nemirovski \cite{nemirovski2004prox} showed $\mathcal{O}(1/k)$ convergence rate for the mirror-prox algorithm (a special case of which is the EG method) in convex-concave saddle point problems over compact sets. This result was extended to unbounded sets by \cite{monteiro_1} where a different error criterion was used. \cite{nedic2009subgradient} analyzed the (sub)Gradient Descent Ascent (GDA) algorithm for convex-concave saddle point problems when the (sub)gradients are bounded over the constraint set. 

Several papers study the special case of Problem~\eqref{main_prob} when the objective function is of the form $f(\bbx, \bby) = g(\bbx) + \bbx^{\top} \bbA \bby - h(\bby)$, i.e., the cross term is bilinear. For this case, when the functions $g$ and $h$ are strongly convex, primal-dual gradient-type methods converge linearly \citep{chen1997convergence,bauschke2011convex}.
Further, \cite{DBLP:journals/corr/abs-1802-01504} showed that GDA achieves a linear convergence rate when $g$ is convex and $h$ is strongly convex. 
\cite{chambolle2011first} introduced a primal-dual variant of the proximal point method that converges to a saddle point at a sublinear rate 
when  $g$ and $h$ are convex and at a linear rate when $g$ and $h$ are strongly convex.

For the case when $f(\bbx, \bby)$ is strongly concave with respect to $\bby$, but possibly nonconvex with respect to $\bbx$, \cite{sanjabi2018convergence} provided convergence to a first-order stationary point using an algorithm that requires running multiple updates with respect to $\bby$ at each step. Recently, \cite{sanjabi2018solving} extended this result to the setting when $f$ is Polyak-Lojasiewicz with respect to $\bby$. 

There are several papers which solve  stochastic version of Problem \eqref{main_prob}, i.e., the case where one does not have access to the exact gradients of the function, but in fact an unbiased estimate of it. Papers including \cite{nemirovski2009robust, juditsky2011solving, chen2014optimal} solve this problem in the case where the objective function is convex in $\bbx$ and concave in $\bby$. More recently \cite{palaniappan2016stochastic} uses a variance reduced version of the proximal gradient method and \cite{chavdarova2019reducing} uses a variance reduced version of the EG method to solve Problem \eqref{main_prob} when the function is strongly convex in $\bbx$ and strongly concave in $\bby$ and the function has a finite sum structure.

Optimistic gradient methods have also been studied in the context of convex online learning. In particular, \cite{DBLP:conf/colt/RakhlinS13,rakhlin2013optimization} introduced the general version of the Optimistic Mirror Descent algorithm in the framework of online optimization. Prior to this work, a special case of Optimistic Mirror descent was analyzed by \cite{DBLP:journals/jmlr/ChiangYLMLJZ12}, again in the context of online learning.

\begin{table}[t!]
\centering
\captionsetup{justification=centering}
\begin{tabular}{ c|c|c|c } 
 \hline
 Reference & Assumptions on $f(\bbx, \bby)$ & Rate (EG) & Rate (OGDA) \\ 
 \hline \hline
 \cite{DBLP:journals/corr/abs-1802-06132}  & Bilinear & $\mathcal{O} (\kappa^2 \log(1/\epsilon))$ & $\mathcal{O} (\kappa \log(1/\epsilon))$ \\ 
 \hline
  \cite{DBLP:journals/corr/abs-1802-06132}  & Strongly Convex-Strongly Concave & \xmark & \xmark \\ 
 \hline
  \cite{gidel2018variational} & Bilinear  & \xmark & \xmark \\ 
 \hline
 \cite{gidel2018variational} & Strongly Convex-Strongly Concave  & \xmark & $\mathcal{O} (\kappa \log(1/\epsilon))^{\star}$ \\ 
 \hline
  \cite{tseng_1} & Bilinear & $\mathcal{O} (\kappa \log(1/\epsilon))$ & \xmark \\
 \hline
 \cite{tseng_1} & Strongly Convex-Strongly Concave & $\mathcal{O} (\kappa \log(1/\epsilon))$ & \xmark \\
 \hline \hline \hline
 This paper & Bilinear & $\mathcal{O} (\kappa \log(1/\epsilon))$ & $\mathcal{O} (\kappa \log(1/\epsilon))$ \\
 \hline
This paper & Strongly Convex-Strongly Concave & $\mathcal{O} (\kappa \log(1/\epsilon))$ & $\mathcal{O} (\kappa \log(1/\epsilon))$ 
\vspace{0.33cm}
\end{tabular}
\vspace{-2mm}
\caption{Comparison of rates in different papers \\($^{\star}$- \cite{gidel2018variational}  shows the convergence of the half points and not the original OGDA iterates.)
}\vspace{-5mm}
\label{table:1}
\end{table}

\textbf{Outline.} The rest of the paper is organized as follows. We start the paper by presenting some definitions and preliminaries required for presenting our results in Section~\ref{sec:preliminaries}. Then, we revisit the Proximal Point (PP) point method in Section~\ref{sec:prox_point_method} and present its convergence properties for bilinear  (Theorem~\ref{thm:bilinear_ppm}) and general strongly convex-strongly concave (Theorem~\ref{thm:scvx_sccv_ppm}) problems . 
In Section~\ref{sec:OGD}, we show that the Optimistic Gradient Descent Ascent (OGDA) is an approximation of PP (Proposition~\ref{prop:OGDA_Approx}) and prove its linear convergence rate for bilinear (Theorem~\ref{thm:bili}) and strongly convex-strongly concave (Theorem~\ref{thm:scv})  problems. We generalize the OGDA method in terms of its parameters and show the convergence of the generalized OGDA method for the bilinear case (Theorem \ref{thm:bili_g}). 
In Section~\ref{sec:EG}, we recap the update of Extra-gradient (EG) method for solving a saddle point problem. Then, we show that EG can be interpreted as an approximation of PP (Proposition~\ref{prop:EG_Approx}) and use this interpretation to study the convergence properties of EG in bilinear problems (Theorem~\ref{thm:EG_bilinear}) and general strongly convex-strongly concave problems (Theorem~\ref{thm:EG_scv}). 
We close the paper with concluding remarks. Due to space limitation our numerical experiments and proofs are presented in the supplementary material. 

\textbf{Notation.} Lowercase boldface $\bbv$ denotes a vector and uppercase boldface $\bbA$ denotes a matrix. We use $\|\bbv\|$ to denote the Euclidean norm of vector $\bbv$. Given a multi-input function $f(\bbx,\bby)$, its gradient with respect to $\bbx$ and $\bby$ at $(\bbx_0,\bby_0)$ are denoted by $\nabla_\bbx f(\bbx_0,\bby_0)$ and $\nabla_\bby f(\bbx_0,\bby_0)$, respectively. We refer to the largest and smallest eigenvalues of a matrix $\bbA$ by $\lambda_{\max}(\bbA)$ and  $\lambda_{\min}(\bbA)$, respectively.

%% file: preliminaries.tex

\section{Preliminaries}\label{sec:preliminaries}

In this section we present properties and notations used in our results.

\begin{definition} \label{def:lips_grad}
A function $\phi: \reals^n\to\reals$ is $L$-smooth if it has $L$-Lipschitz continuous gradients on $\reals^n$, i.e., for any $\bbx, \hbx \in \reals^n$, we have 
$|| \nabla \phi(\bbx) - \nabla \phi(\hbx) || \leq L ||\bbx - \hbx||$.
\end{definition}

%

\begin{definition}  \label{def:strong_convex}
A continuously differentiable function $\phi: \reals^n\to\reals$ is $\mu$-strongly convex on $\mathbb{R}^n$ if for any $\bbx, \hbx \in \reals^n$, we have
$\phi(\hbx) \geq \phi(\bbx) + \nabla \phi(\bbx) ^T (\hbx - \bbx) + \frac{\mu}{2}  ||\hbx - \bbx||^2.$
 Further, $\phi(\bbx)$ is $\mu$-strongly concave if $-\phi(\bbx)$ is $\mu$-strongly convex. If we set $\mu=0$, then we recover the definition of convexity for a continuous differentiable function.
\end{definition} 

\begin{definition}  \label{def:saddle_point}
The pair $(\bbx^*, \bby^*)$ is a saddle point of a convex-concave function $f: \reals^n\times \reals^m \to \reals$, if for any $\bbx \in \reals^n$ and $ \bby \in \reals^m$, we have
$f(\bbx^*, \bby) \leq f(\bbx^*, \bby^*) \leq f(\bbx, \bby^*).$
\end{definition} 


Throughout the paper, we consider two specific cases of Problem \eqref{main_prob} stated in the next set of assumptions.

\begin{assumption}
\label{ass:bili}
The function $f(\bbx, \bby)$ is a bilinear function of the form 
$f(\bbx, \bby) = \bbx^{\top} \bbB \bby$, where $\bbB \in \reals^{d \times d}$ is a square full-rank matrix. The point $(\bbx^*, \bby^*) = (\bb0, \bb0)$ is the unique saddle point. In this case, we define the condition number of the problem as $\kappa := \frac{\lambda_{\max}(\bbB^\top  \bbB)}{\lambda_{\min}(\bbB^\top \bbB)}$.
\end{assumption}

\begin{assumption}
\label{ass:scsc_1}
The function $f(\bbx, \bby)$ is continuously differentiable in $\bbx$ and $\bby$. Further, $f$ is $\mu_x$-strongly convex in $\bbx$ and $\mu_y$-strongly concave in $\bby$. The unique saddle point of $f(\bbx,\bby)$ is denoted by $(\bbx^*, \bby^*)$. We define $\mu = \min \{ \mu_x, \mu_y \}$.
\end{assumption}


\begin{assumption}
\label{ass:scsc}
The gradient $\nabla_{\bbx}f(\bbx, \bby)$, is $L_x$-Lipschitz in $\bbx$ and $L_{xy}$-Lipschitz in $\bby$, i.e., \begin{align*}
\| \nabla_{\bbx}f(\bbx_1, \bby) - \nabla_{\bbx}f(\bbx_2, \bby)\| &\leq L_x\| \bbx_1 - \bbx_2 \| \quad \forall \: \bby,  \\
\| \nabla_{\bbx}f(\bbx, \bby_1) - \nabla_{\bbx}f(\bbx, \bby_2)\| &\leq L_{xy}\|\bby_1 - \bby_2 \| \quad {\forall} \: \bbx.
\end{align*}
Moreover, the gradient $\nabla_{\bby}f(\bbx, \bby)$, is $L_y$-Lipschitz in $\bby$ and $L_{yx}$-Lipschitz in $\bbx$, i.e., \begin{align*}
\| \nabla_{\bby}f(\bbx, \bby_1) - \nabla_{\bby}f(\bbx, \bby_2)| &\leq L_y\| \bby_1 - \bby_2 \| \quad {\forall} \: \bbx, \\
\| \nabla_{\bby}f(\bbx_1, \bby) - \nabla_{\bby}f(\bbx_2, \bby)\| &\leq L_{yx}\|\bbx_1 - \bbx_2 \| \quad {\forall} \: \bby.
\end{align*}
We define $L = \max \{ L_x, L_{xy}, L_y, L_{yx} \}$.
\end{assumption}

\begin{remark}\label{cond_numb_remark}
Under Assumptions \ref{ass:scsc_1} and \ref{ass:scsc},  we define the condition number of the problem as $\kappa := L/\mu$.
\end{remark}


In the following sections, we present and analyze three different iterative algorithms for solving the saddle point problem introduced in~\eqref{main_prob}. The $k$-th iterates of any of these algorithms are denoted by $(\bbx_k, \bby_k)$. We denote 
\begin{equation}\label{r_k_definition}
r_k = \| \bbx_k - \bbx^*\|^2 + \|\bby_k - \bby^* \|^2,
\end{equation}
as the distance to the saddle point $(\bbx^*, \bby^*)$ at iteration $k$.

%% file: proximal.tex


\section{Proximal Point method}\label{sec:prox_point_method}

We start our analysis by Proximal Point (PP) method, which will serve as a benchmark for the analysis of Extra-gradient and Optimistic Gradient Descent Ascent methods.
The update of PP method for minimizing a convex function $h$ is defined as
\begin{equation}\label{eq:prox_point_org_update}
\bbx_{k+1} = \text{prox}_{\frac{1}{\eta},h} (\bbx_k) = \argmin\left\{ h(\bbx) + \frac{1}{2\eta}\|\bbx-\bbx_k\|^2\right\},
\end{equation}
where $\eta$ is a positive scalar \citep{bertsekas1999nonlinear,beck2017first}. Using the optimality condition of the update in \eqref{eq:prox_point_org_update}, one can also write the update of the PP method as 
$\bbx_{k+1} = \bbx_k - \eta \nabla h (\bbx_{k+1})$.
This expression shows that the PP method is an \textit{implicit algorithm}. 
Convergence properties of PP for convex minimization have been extensively studied \citep{rockafellar1976augmented,guler1991convergence,ferris1991finite,eckstein1992douglas,parikh2014proximal,beck2017first}. The extension of PP for solving saddle point problems has been also studied in~
\cite{rockafellar1976monotone}.
Here, we recap the update of PP for solving the min-max problem in \eqref{main_prob}. To do so, we define the iterates $\{\bbx_{k+1},\bby_{k+1}\} $ as the unique solution to the saddle point problem
\begin{align}\label{eq:prox_point_min_max_update_0}
\min_{\bbx \in \reals^m} \max_{\bby \in \reals^n} \left\{ f(\bbx,\bby) + \frac{1}{2\eta}\|\bbx-\bbx_k\|^2- \frac{1}{2\eta}\|\bby-\bby_k\|^2\right\}.
\end{align}
Using the optimality conditions of \eqref{eq:prox_point_min_max_update_0}
(which are necessary and sufficient since the problem in \eqref{eq:prox_point_min_max_update_0}
is convex), the update of the PP method for the saddle point problem in \eqref{main_prob} can be written as 
\begin{align}\label{eq:prox_point_min_max_update_final}
\bbx_{k+1}=  \bbx_k - \eta \nabla_{\bbx} f(\bbx_{k+1},\bby_{k+1}),\qquad
\bby_{k+1} =  \bby_k + \eta \nabla_{\bby} f(\bbx_{k+1},\bby_{k+1}).
\end{align}


Note that implementing the system of updates in \eqref{eq:prox_point_min_max_update_final} requires computing the operators $(\bbI+\eta \nabla_\bbx f)^{-1}$ and $(\bbI+ \eta \nabla_\bby f)^{-1}$, and, therefore, may not be computationally affordable for any general function $f$. 

%
In the following theorem, we show that the PP method converges linearly to $(\bbx^*,\bby^*)=(\bb0,\bb0)$ which is the unique solution of the problem $\min_\bbx\max_\bby\bbx^\top\bbB \bby$. This result was established in Theorem 2 of \citep{rockafellar1976monotone} and we mention it here for completeness and we later use it as a benchmark.

%
\begin{theorem}\label{thm:bilinear_ppm}
Consider the saddle point problem in \eqref{main_prob} under Assumption \ref{ass:bili} and the proximal point method in \eqref{eq:prox_point_min_max_update_final}. Further, recall the definition of $r_k$ in \eqref{r_k_definition}. Then, for any $\eta>0$, the iterates $\{\bbx_k,\bby_k\}_{k\geq 0}$ generated by the proximal point method satisfy 
\begin{equation}
r_{k+1} \leq \frac{1}{1\!+\!\eta^2\lambda_{min}(\bbB^\top \bbB)} r_k. \nonumber
\end{equation}
\end{theorem}

%

%

In the following theorem, we characterize the convergence rate of PP for a function $f(\bbx,\bby)$ that is strongly convex with respect to $\bbx$ and strongly concave with respect to $\bby$. Once again, this result was established in \citep{rockafellar1976monotone} and we mention it here for completeness and we later use it as a benchmark.
%
\begin{theorem}\label{thm:scvx_sccv_ppm}
Consider the saddle point problem in \eqref{main_prob} under Assumption \ref{ass:scsc_1} and the proximal point method in \eqref{eq:prox_point_min_max_update_final}. Further, recall the definition of $r_k$ in \eqref{r_k_definition}. Then, for any $\eta>0$, the iterates $\{\bbx_k,\bby_k\}_{k\geq 0}$ generated by the proximal point method satisfy
\begin{align}
r_{k+1} \leq \frac{1}{1+\eta\mu} r_k. \nonumber
\end{align}
\end{theorem}

%
Theorem~\ref{thm:scvx_sccv_ppm} states that for the general saddle point problem in \eqref{main_prob}, if the function is strongly convex-strongly concave, the iterates generated by the PP method converge linearly to the optimal solution. 


%% file: ogd.tex

\section{Optimistic Gradient Descent Ascent method}\label{sec:OGD}

In this section, we study the Optimistic Gradient Descent Ascent (OGDA) method for solving saddle point problems. We first show that OGDA can be considered as an approximation of the proximal point method. Then, we use this interpretation to analyze its convergence properties for bilinear and strongly convex-strongly concave settings. The proximal point approximation approach also allows us to generalize the update of OGDA as we discuss in detail in Section~\ref{ogda_generalized}.

\subsection {Convergence rate of the OGDA Method}
The main idea behind the updates of the OGDA method is the addition of a ``negative-momentum" term to the updates which can be clearly seen when we write the iterations as follows:
\begin{align}
\bbx_{k+1} &= \bbx_k - \eta \nabla_{\bbx} f(\bbx_{k},\bby_{k})  - \eta\left(\nabla_{\bbx} f(\bbx_{k},\bby_{k}) - \nabla_{\bbx} f(\bbx_{k-1},\bby_{k-1})\right), \nonumber \\
\bby_{k+1} &= \bby_k + \eta \nabla_{\bby} f(\bbx_{k},\bby_{k}) + \eta\left(\nabla_{\bby} f(\bbx_{k},\bby_{k}) - \nabla_{\bby} f(\bbx_{k-1},\bby_{k-1})\right). \nonumber
\end{align}
The last term in parenthesis for each of the updates can be interpreted as a ``negative-momentum", differentiating the OGDA method from vanilla Gradient Descent Ascent (GDA). 

%
\begin{algorithm}[tb]
\caption{OGDA method for saddle point problems}\label{alg:OGDA} 
\begin{algorithmic}[1] 
{\REQUIRE Stepsize $\eta>0$, vectors $\bbx_{-1},\bby_{-1}, \bbx_0,\bby_0\in \reals^d$ 
 \vspace{0.5mm}
\FOR {$k=1,2,\ldots$}
\vspace{0.5mm}
     \STATE$\displaystyle{\bbx_{k+1} =  \bbx_k - 2\eta \nabla_{\bbx} f(\bbx_{k},\bby_{k})  + \eta \nabla_{\bbx} f(\bbx_{k-1},\bby_{k-1})}$;
   \vspace{0.5mm}
    \STATE$\displaystyle{\bby_{k+1} =  \bby_k + 2\eta \nabla_{\bby} f(\bbx_{k},\bby_{k}) - \eta \nabla_{\bby} f(\bbx_{k-1},\bby_{k-1})}$;
     \vspace{0.5mm}
\ENDFOR}
\end{algorithmic}\end{algorithm}

%

We analyze the OGDA method as an approximation of the Proximal Point (PP) method presented in Section \ref{sec:prox_point_method}. We first focus on the bilinear case (Assumption \ref{ass:bili}) for which the OGDA updates are
\begin{align}
\bbx_{k+1} =  \bbx_k - 2 \eta \bbB \bby_k  + \eta \bbB \bby_{k-1}, \qquad 
\bby_{k+1} =  \bbx_k + 2 \eta \bbB^{\top} \bbx_k  + \eta \bbB^{\top} \bbx_{k-1} .\nonumber 
\end{align}
Note that the update of the PP method for the variable $\bbx$ in the considered bilinear problem is 
\begin{align}
\bbx_{k+1} &= (I + \eta^2 \bbB \bbB^{\top})^{-1} (\bbx_k - \eta \bbB \bby_k) \nonumber \\
&= (I - \eta^2 \bbB \bbB^{\top} + o(\eta^2)) (\bbx_k - \eta \bbB \bby_k) \nonumber \\
& = \bbx_k - \eta \bbB \bby_k - \eta \bbB (\eta \bbB^{\top} \bbx_k - \eta^2 \bbB^{\top} \bbB \bby_k) + o(\eta^2),\nonumber
\end{align}
where we used the fact that $I - \eta^2 \bbB \bbB^{\top}$ is an approximation of $(I + \eta^2 \bbB \bbB^{\top})^{-1} $ with an error of $o(\eta^2)$. Regrouping the terms and using the updates of the PP method yield 
\begin{align}
\bbx_{k+1} &= \bbx_k- 2 \eta \bbB \bby_k
- \eta \bbB (\eta \bbB^{\top} \bbx_k - (1 + \eta^2 \bbB^{\top} \bbB)\bby_k) + o(\eta^2)\nonumber \\
&= \bbx_k - 2 \eta \bbB \bby_k  - \eta \bbB(\eta \bbB^{\top} \bbx_k - \bby_{k-1} - \eta \bbB^{\top} \bbx_{k-1}) + o(\eta^2)\nonumber \\
&  = \bbx_k - 2 \eta \bbB \bby_k  + \eta \bbB \bby_{k-1}  + o(\eta^2) \nonumber,
\end{align} 
where the last expression is the OGDA update for variable $\bbx$ plus an additional error of $o(\eta^2)$. A similar derivation can be done for the update of variable $\bby$ to show that OGDA is an approximation of the PP method up to $o(\eta^2)$. In the following proposition, we show that this observation can be generalized for any general smooth (possibly nonconvex) function $f(\bbx,\bby)$.

%
\begin{proposition}\label{prop:OGDA_Approx}
Consider the saddle point problem in \eqref{main_prob}. Given a point $(\bbx_k, \bby_k)$, let $(\hat{\bbx}_{k+1}, \hat{\bby}_{k+1})$ be the point we obtain by performing the PP update on $(\bbx_k, \bby_k)$, and let $(\bbx_{k+1}, \bby_{k+1})$   be the point we obtain by performing the OGDA update on $(\bbx_k, \bby_k)$.
Then, for a given stepsize $\eta>0$ we have
\begin{align}
\| \bbx_{k+1} - \hat{\bbx}_{k+1} \|\leq o(\eta^2), \qquad   \| \bby_{k+1} - \hat{\bby}_{k+1} \| \leq o(\eta^2).
\end{align}
\end{proposition}



To analyze the convergence of OGDA, we view it as a proximal point algorithm with an additional error term. In the following theorem, we characterize the convergence rate of the OGDA method for the bilinear saddle point problem defined in Assumption \ref{ass:bili}. 

%
\begin{theorem}[Bilinear case]\label{thm:bili}
Consider the saddle point problem in \eqref{main_prob} under Assumption \ref{ass:bili} and the OGDA method outlined in Algorithm~\ref{alg:OGDA}. Further, recall the definition of $r_k$ in \eqref{r_k_definition}. If we set $\eta =(1/{40\sqrt{\lambda_{\max}(\bbB^\top  \bbB)}})$, then the iterates $\{\bbx_k,\bby_k\}_{k\geq 0}$ generated by the OGDA method satisfy 
\begin{equation}
r_{k+1} \leq   \left(1-c \kappa^{-1}\right)^{k} \hat{r}_0, \nonumber
\end{equation}
where $\hat{r}_0 = \max \{ r_2, r_1, r_0 \}$ and $c$ is a positive constant independent of the problem parameters.  
\end{theorem}


The result in Theorem~\ref{thm:bili} shows linear convergence of OGDA in a bilinear problem of the form $f(\bbx,\bby)=\bbx^\top\bbB\bby$ where  matrix $\bbB$ is square and full rank. It further shows that the overall number of iterations to obtain an $\eps$-accurate solution is of $\mathcal{O}(\kappa \log(1/\eps))$, where $\kappa$ is the problem condition number as defined in Assumption \ref{ass:bili}. We would like to mention that this result is similar to the one shown in \citep{DBLP:journals/corr/abs-1802-06132}, except here we analyze OGDA as an approximation of PP. 

In the following theorem, we again use the proximal point approximation interpretation of OGDA to provide a convergence rate estimate for this algorithm when it is used for solving a general strongly convex-strongly concave saddle point problem.

%
\begin{theorem}[Strongly convex-strongly concave case]\label{thm:scv}
Consider the saddle point problem in \eqref{main_prob} under Assumptions \ref{ass:scsc_1} and \ref{ass:scsc} and the OGDA method outlined in Algorithm~\ref{alg:OGDA}. Further, recall the definition of $r_k$ in \eqref{r_k_definition}. If we set $\eta = (1/(4 L))$, then the iterates $\{\bbx_k,\bby_k\}_{k\geq 0}$ generated by OGDA satisfy 
\begin{align}
r_{k+1} \leq \left( 1 -  c \kappa^{-1} \right) ^{k} \hat{r}_0 \nonumber,
\end{align}
where $\hat{r}_0 = c_1 \kappa^2 r_0$ and $c, c_1$ are positive constant independent of the problem parameters.  
\end{theorem}

%

The result in Theorem~\ref{thm:scv} shows that OGDA converges linearly to the optimal solution under the assumptions that $f$ is smooth and strongly convex-strongly concave. In other words, it shows that to achieve a point with error $r_k\leq \eps$, we need to run at most $\mathcal{O}(\kappa \log(1/\eps))$ iterations of OGDA.
This result can be compared with the results in \citep{gidel2018variational}, a recent independent work which derived the OGDA updates as a variant of the EG updates (interpreting OGDA as extrapolation from the past). In this connection, the OGDA iterates are the ``midpoints" 
whereas \citep{gidel2018variational} provides a convergence of the original points (not the OGDA iterates) to an error of $\eps$ in $\mathcal{O}(\kappa \log(1/\eps))$.

\subsection{Generalized OGDA method}\label{ogda_generalized}

The update of OGDA both in theory and practice is only studied for the case that the coefficients of both $\nabla_{\bbx} f(\bbx_{k},\bby_{k})$ and $\nabla_{\bbx} f(\bbx_{k},\bby_{k}) - \nabla_{\bbx} f(\bbx_{k-1},\bby_{k-1})$ are $\eta$. This implies that in the OGDA update at step $k$, the coefficient of the current gradient, i.e., $\nabla_{\bbx} f(\bbx_{k},\bby_{k})$, should be exactly twice the coefficient of the negative of the previous gradient, i.e., $-\nabla_{\bbx} f(\bbx_{k},\bby_{k})$. It has been an open question to see if different stepsizes can be used for these terms. In this section, we generalize OGDA where the coefficients for the gradient descent and the negative momentum terms are not necessary equal to each other. We consider the following OGDA dynamics with general stepsize parameters $\alpha, \beta > 0$:
\begin{align}
\bbx_{k+1} &= \bbx_k - (\alpha +\beta) \nabla_{\bbx} f(\bbx_{k}, \bby_{k}) + \beta \nabla_{\bbx} f(\bbx_{k-1}, \bby_{k-1}), \label{eq:OGDA_g_1} \\
\bby_{k+1} &= \bby_k + (\alpha +\beta) \nabla_{\bby} f(\bbx_{k}, \bby_{k}) - \beta \nabla_{\bby} f(\bbx_{k-1}, \bby_{k-1}). \label{eq:OGDA_g_2}
\end{align}
Note that for $\alpha = \beta$, we recover the original OGDA method. Our goal is to show that  OGDA is convergent even if $\alpha$ and $\beta$ are not equal to each other, as long as their difference is sufficiently small. In the following theorem, we formally state our result for the generalized OGDA method described in \eqref{eq:OGDA_g_1} and \eqref{eq:OGDA_g_2} when the objective function $f$ has a bilinear form of $f(\bbx,\bby)=\bbx^\top\bbB\bby$.


\begin{theorem}[Generalized bilinear case]\label{thm:bili_g}
Consider the saddle point problem in \eqref{main_prob} under Assumption~\ref{ass:bili} and the generalized OGDA method in \eqref{eq:OGDA_g_1}-\eqref{eq:OGDA_g_2}. Further, recall the definition of $r_k$ in~\eqref{r_k_definition}. 
If we set $\alpha = {1}/({40\sqrt{\lambda_{\max}(\bbB^\top  \bbB)}})$ and $\alpha$ and $\beta$ satisfy the conditions $0 < \alpha - K \alpha^2 \leq \beta \leq \alpha$ for some constant $K>0$, then the iterates $\{\bbx_k,\bby_k\}_{k\geq 0}$ generated by the generalized OGDA method satisfy
\begin{equation}
r_{k+1} \leq  \left( 1 -  c \kappa^{-1} \right)^{k} \ \hat{r}_0, \nonumber
\end{equation}
where $\hat{r}_0 = \max \{ r_2, r_1, r_0\}$ and $c$ is a positive constant independent of the problem parameters.

%
\end{theorem}

Theorem~\ref{thm:bili_g} shows that it is not necessary to use a factor of $2$ in the OGDA update to have a linearly convergent method and for a wide range of parameters this result holds. A result similar to Theorem~\ref{thm:bili_g} can be established when $\beta > \alpha$. We do not state the results here due to space limitations. 

%% file: extragradient.tex
\section{Extra-gradient method}\label{sec:EG}

In this section, we study the Extra-gradient (EG) method for solving the general saddle point problem in \eqref{main_prob} and provide linear rates of convergence  for the bilinear and the strongly convex-strongly concave case by interpreting this algorithm as an approximation of the proximal point method. 


The main idea of the EG method is to use the gradient at the current point to find a mid-point, and then use the gradient at that mid-point to find the next iterate. To be more precise, given a stepsize $\eta>0$, the update of EG at step $k$ for solving the saddle point problem in \eqref{main_prob} has two steps. First, we find mid-point iterates $\bbx_{k+1/2} $ and $\bby_{k+1/2} $ by performing a primal-dual gradient update as
\begin{align}
\bbx_{k+1/2}=  \bbx_k - \eta \nabla_{\bbx} f(\bbx_{k},\bby_{k}), \qquad 
\bby_{k+1/2} =  \bby_k + \eta \nabla_{\bby} f(\bbx_{k},\bby_{k}). \nonumber
\end{align}
 Then, the gradients evaluated at the midpoints $\bbx_{k+1/2} $ and $\bby_{k+1/2} $ are used to compute the new iterates  $\bbx_{k+1} $ and $\bby_{k+1} $ by performing the updates
\begin{align}
\bbx_{k+1}  =  \bbx_k - \eta \nabla_{\bbx} f(\bbx_{k+1/2},\bby_{k+1/2}),\qquad
\bby_{k+1} =  \bby_k + \eta \nabla_{\bby} f(\bbx_{k+1/2},\bby_{k+1/2}). \nonumber
\end{align}
 The steps of the EG method for solving saddle point problems are outlined in Algorithm~\ref{alg:EG}.

 Note that in the update of the EG method, as the name suggests, requires evaluation of extra gradients at the midpoints $\bbx_{k+1/2} $ and $\bby_{k+1/2} $ which doubles the computational complexity of EG compared to the vanilla Gradient Descent Ascent (GDA) method. 
 We show next EG approximates the Proximal Point (PP) method more accurately, as compared to the GDA method.
Consider the bilinear saddle point problem defined in Assumption \ref{ass:bili}. By following the update of PP in Section~\ref{sec:prox_point_method} and simplifying the expressions, the PP update for the bilinear problem under Assumption \ref{ass:bili}
 can be written as 
  \begin{align}
\bbx_{k+1} = (\bbI + \eta^2 \bbB \bbB^\top)^{-1} (\bbx_k - \eta \bbB \bby_k), \qquad
\bby_{k+1} = (\bbI + \eta^2 \bbB^\top \bbB)^{-1} (\bby_k + \eta \bbB^\top \bbx_k). \nonumber
\end{align}
   As the computation of the inverse $ (\bbI + \eta^2 \bbB^\top \bbB)^{-1}$ could be costly, one can use $\bbI$ instead with an error of  $o(\eta)$. This approximation retrieves the update of GDA which is known to possibly diverge for bilinear saddle point problems
   (see \cite{DBLP:journals/corr/abs-1711-00141}). If we use the more accurate approximation $ (I + \eta^2 \bbB \bbB^\top)^{-1}\approx (\bbI-\eta^2\bbB \bbB^\top)$ which has an error of $o(\eta^2)$, we obtain 
  \begin{align}
\bbx_{k+1} &= (\bbI - \eta^2 \bbB \bbB^\top + o(\eta^2)) (\bbx_k - \eta \bbB \bby_k)= \bbx_k - \eta \bbB \bby_k -\eta^2 \bbB \bbB^\top \bbx_k +o(\eta^2), \label{eq:extra_bilinear_x} \\
\bby_{k+1} &= (\bbI - \eta^2 \bbB^\top \bbB+o(\eta^2)) (\bby_k + \eta \bbB^\top \bbx_k)= \bby_k + \eta \bbB^\top \bbx_k -\eta^2 \bbB^\top \bbB\bby_k +o(\eta^2). \label{eq:extra_bilinear_y}
\end{align}
If we ignore the extra terms in \eqref{eq:extra_bilinear_x}-\eqref{eq:extra_bilinear_y} which are of $o(\eta^2)$, we recover the update of the EG method for the bilinear saddle point problem defined in Assumption \ref{ass:bili}. Therefore, in the bilinear problem, the EG method can be interpreted as an approximation of the PP method with an error of $o(\eta^2)$. 
In the following proposition, we extend this result and show that for any general smooth (possibly nonconvex) function $f(\bbx,\bby)$, EG is an $o(\eta^2)$ approximation of PP.

%
\begin{algorithm}[tb]
\caption{Extra-gradient method for saddle point problem}\label{alg:EG} 
\begin{algorithmic}[1] 
{\REQUIRE Stepsize $\eta>0$, initial vectors $\bbx_0,\bby_0\in \reals^d$ 
 \vspace{0.5mm}
\FOR {$k=1,2,\ldots$}
\vspace{0.5mm}
   \STATE  Compute $\displaystyle{\bbx_{k+1/2} =  \bbx_k - \eta \nabla_{\bbx} f(\bbx_{k},\bby_{k})}$ and  $\displaystyle{\bby_{k+1/2} =  \bby_k + \eta \nabla_{\bby} f(\bbx_{k},\bby_{k})}$;
   \STATE  Update $\displaystyle{\bbx_{k+1} =  \bbx_k - \eta \nabla_{\bbx} f(\bbx_{k+1/2},\bby_{k+1/2})}$ and $\displaystyle{\bby_{k+1} =  \bby_k + \eta \nabla_{\bby} f(\bbx_{k+1/2},\bby_{k+1/2})}$;
\ENDFOR}
\end{algorithmic}\end{algorithm}

%


%
\begin{proposition}\label{prop:EG_Approx}
Consider the saddle point problem in \eqref{main_prob}. Given a point $(\bbx_k, \bby_k)$, let $(\hat{\bbx}_{k+1}, \hat{\bby}_{k+1})$ be the point we obtain by performing the PP update on $(\bbx_k, \bby_k)$, and let $(\bbx_{k+1}, \bby_{k+1})$   be the point we obtain by performing the EG update on $(\bbx_k, \bby_k)$. Then, for a given stepsize $\eta>0$ we have 
\begin{align}
  \| \bby_{k+1} - \hat{\bby}_{k+1} \| \leq o(\eta^2), \qquad \| \bbx_{k+1} - \hat{\bbx}_{k+1} \|  \leq o(\eta^2).
\end{align}
\end{proposition}


The next theorem views the EG method as the PP method with an error and properly bounds the error to provide convergence rate estimates for the EG method in the bilinear case. 


%
\begin{theorem}[Bilinear case]\label{thm:EG_bilinear}
Consider the saddle point problem in \eqref{main_prob} under Assumption \ref{ass:bili} and the extra-gradient (EG) method outlined in Algorithm~\ref{alg:EG}. Further, recall the definition of $r_k$ in~\eqref{r_k_definition}. If we set $\eta = {1}/({2\sqrt{2\lambda_{\max}(\bbB^\top  \bbB)}})$, then the iterates $\{\bbx_k,\bby_k\}_{k\geq 0}$ generated by the EG method satisfy 
\begin{equation}\label{thm:EG_bilinear_claim}
r_{k+1} \leq  \left(1-c \kappa^{-1}\right) \ r_k,
\end{equation}
where $c$ is a positive constant independent of the problem parameters.  
\end{theorem}

 The result in Theorem~\ref{thm:EG_bilinear} shows linear convergence of the iterates generated by the EG method for a bilinear problem of the form $f(\bbx,\bby)=\bbx^\top\bbB\bby$ where the matrix $\bbB$ is square and full rank. In other words, we obtain that the overall number of iterations to reach a point satisfying $\|\bbx_{k}\|^2+\|\bby_{k}\|^2\leq \epsilon$ is at most $\mathcal{O}(\kappa \log(1/\epsilon))$ which matches the rate achieved in \cite{tseng_1} for bilinear problems. It is worth mentioning that we obtain this optimal complexity of $\mathcal{O}(\kappa \log(1/\epsilon))$ for EG in bilinear problems by analyzing this algorithm as an approximation of the PP method which differs from the approach used in \citep{tseng_1} that directly analyzes EG using a variational inequality approach. We further would like to add that this result improves the $\mathcal{O}(\kappa^2 \log(1/\epsilon))$ of \citep{DBLP:journals/corr/abs-1802-06132} for EG in bilinear problems.

The following theorem characterizes the convergence rate of the EG method when $f(\bbx, \bby)$ is strongly convex-strongly concave.
%


\begin{theorem}[Strongly convex-strongly concave case]\label{thm:EG_scv}
Consider the saddle point problem in \eqref{main_prob} under Assumptions \ref{ass:scsc_1} and \ref{ass:scsc} and the extra-gradient (EG) method outlined in Algorithm~\ref{alg:EG}. Further, recall the definition of the condition number $\kappa$ in Remark~\ref{cond_numb_remark} and the definition of $r_k$ in~\eqref{r_k_definition}. If we set $\eta = {1}/({4L})$, 
then the iterates $\{\bbx_k,\bby_k\}_{k\geq 0}$ generated by the EG method satisfy 
\begin{align}
r_{k+1} \leq \left(1- c \kappa^{-1} \right) r_k \nonumber,
\end{align}
where $c$ is a positive constant independent of the problem parameters.  
\end{theorem}
The result in Theorem~\ref{thm:EG_scv} characterizes a linear convergence rate for the EG algorithm in a general smooth and strongly convex-strongly concave case. Similar to the bilinear case, our proof relies on interpreting EG as an approximation of the PP method. Theorem~\ref{thm:EG_scv} further shows that the computational complexity of EG to achieve an $\epsilon$-suboptimal solution, i.e., $\|\bbx_{k+1}-\bbx^*\|^2+\|\bby_{k+1}-\bby^*\|^2\leq \eps$, is $\mathcal{O}(\kappa\log(1/\epsilon))$, where $\kappa=L/\mu$ is the condition number of the number. Note that this complexity bound can also be obtained from the results in \citep{tseng_1}.

%% file: simulations.tex
\begin{figure}[t!]
  \centering
\includegraphics[width=0.7\columnwidth]{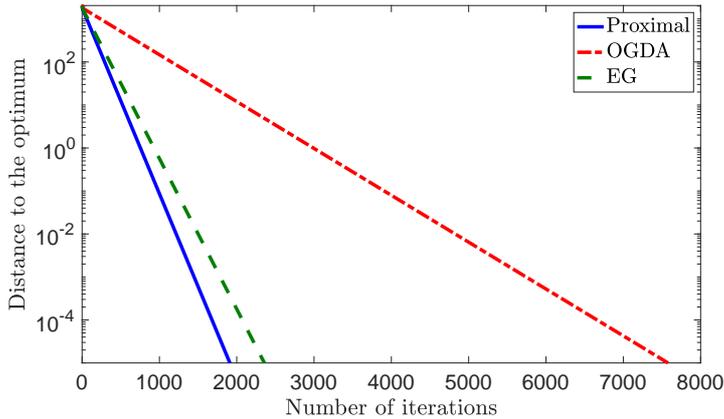}

    \caption{Convergence of proximal point (PP), extra-gradient (EG), and optimistic gradient descent ascent (OGDA) in terms of number of iterations for the bilinear problem in \eqref{exp_bilinear}. All algorithms converge linearly, and the proximal point method has the best performance. Stepsizes of EG and OGDA were tuned for best performance.
}
    
 \label{fig_bili}
 \end{figure}

\section{Numerical Experiments}\label{sec:simul}

In this section, we compare the performance of the Proximal Point~(PP) method with the Extra--Gradient~(EG), Gradient Descent Ascent~(GDA), and Optimistic Gradient Descent Ascent~(OGDA) methods.

We first focus on the following bilinear problem
\begin{align}\label{exp_bilinear}
\min_{\bbx \in \reals^d} \max_{\bby \in \reals^d} \bbx^{\top} \bbB \bby.
\end{align}
where we set $\bbB\in \reals^{d\times d}$ to be a diagonal matrix with a condition number of $100$, and we set the dimension of the problem to $d=10$. The iterates are initialized at $\bbx_0= \mathbf{10}$ and $\bby_0=\mathbf{10}$, where $\mathbf{10}$ is a $d$ dimensional vector with all elements equal to $10$. Figure \ref{fig_bili} demonstrates the errors of PP, OGDA, and EG versus number of iterations for this bilinear problem. Note that in this figure we do not show the error of GDA since it diverges for this problem, as illustrated in Figure 1 (For more details check  \cite{DBLP:journals/corr/abs-1711-00141}). We can observe that all the three considered algorithms converge linearly to the optimal solution $(\bbx^*,\bby^*)=(\bb0,\bb0)$. 

\begin{figure}[t!]
  \centering
\includegraphics[width=0.7\columnwidth]{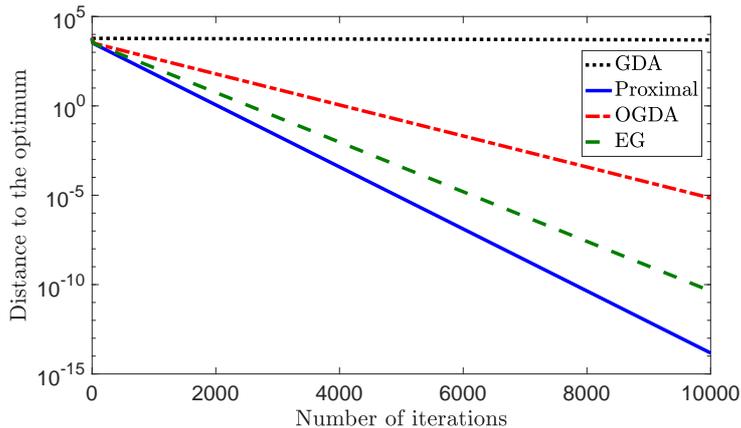}

    \caption{Convergence of proximal point (PP), extra-gradient (EG), optimistic gradient descent ascent (OGDA), and gradient descent ascent (GDA) in terms of number of iterations for the quadratic problem in \eqref{quad_problem}. Stepsizes of EG, OGDA and GDA were tuned for best performance.} 
 \label{fig_quad_iter}
 \end{figure}

We proceed to study the performance of PP, EG, GDA, and OGDA for solving the following  strongly convex-strongly concave saddle point problem
\begin{equation}\label{quad_problem}
\min_{\bbx \in \reals^d} \max_{\bby \in \reals^n}  \frac{1}{n} \left[-\frac{1}{2} \| \bby\|^2 \!-\! \bbb^{\top}\bby\!+\! \bby^{\top} \bbA \bbx \right] + \frac{\lambda}{2} \| \bbx \|^2 
\end{equation}
This is the saddle point reformulation of the linear regression
\begin{align}
\min_{\bbx \in \reals^d} \frac{1}{2n} \| \bbA \bbx - \bbb \|^2
\end{align}
with an $L_2$ regularization, as shown in \cite{DBLP:journals/corr/abs-1802-01504}. As done in \cite{DBLP:journals/corr/abs-1802-01504}, we generate the rows of the matrix $\bbA$ according to a Gaussian distribution $\mathcal{N}(0, I_d)$. Here, we set $d = 50$ and $n = 10$, and assume $\bbb = \bb0$. We also set the regularizer to $\lambda = 1/n$. Figure~\ref{fig_quad_iter} illustrates the distance to the optimal solution of the considered algorithms versus the number of iterations. As we can see, EG and OGDA perform better than GDA and their convergence paths are closer to the one for PP which has the fastest rate. This observation matches our theoretical claim that EG and OGDA are more accurate approximations of PP relative to GDA.

%% file: appendix.tex

\newpage

\section{Supplementary Material}\label{sec:SM}

\subsection{Proof of Theorem~\ref{thm:bilinear_ppm}}

First, note that the update of the Proximal Point (PP) method for the bilinear problem (Assumption \ref{ass:bili}) can be written as
\begin{align}
\bbx_{k+1} &= \bbx_k - \eta \bbB \bby_{k+1}, \label{bilinear_prox_proof_100_1}\\
\bby_{k+1} &= \bby_k + \eta \bbB^{\top} \bbx_{k+1}.\label{bilinear_prox_proof_100_2}
\end{align}
We can simplify the above iterations and write them as an explicit algorithm as follows:
\begin{align}
\bbx_{k+1} &= (\bbI + \eta^2 \bbB \bbB^{\top})^{-1} (\bbx_k - \eta \bbB \bby_k), \label{bilinear_prox_proof_200_1}\\
\bby_{k+1} &= (\bbI + \eta^2 \bbB^{\top} \bbB)^{-1} (\bby_k + \eta \bbB^{\top} \bbx_k). \label{bilinear_prox_proof_200_2}
\end{align}
Let us define the symmetric matrices $\bbQ_x = (\bbI + \eta^2 \bbB \bbB^{\top})^{-1}$ and $\bbQ_y = (\bbI + \eta^2 \bbB^{\top} \bbB)^{-1}$. Based on these definitions, and the expressions in \eqref{bilinear_prox_proof_200_1} and \eqref{bilinear_prox_proof_200_2} we can show that the sum $\|\bbx_{k+1}\|^2 + \|\bby_{k+1}\|^2$ can be written as
\begin{align} \label{eq:prox_bilinear_0} 
\|\bbx_{k+1}\|^2 + \|\bby_{k+1}\|^2 = \|\bbQ_x \bbx_k\|^2 + \eta^2\|\bbQ_x \bbB \bby_k\|^2 & + \|\bbQ_y \bby_k\|^2 + \eta^2 \|\bbQ_y \bbB^{\top} \bbx_k\|^2 \nonumber\\ 
&\qquad 
- 2\eta \bbx_k^\top \bbQ_x \bbB\bby_k
+ 2\eta \bby_k^\top \bbQ_y \bbB^\top \bbx_k.
\end{align}
To simplify the expression in \eqref{eq:prox_bilinear_0} we first prove the following lemma which is also useful in the rest of proofs. 


\begin{lemma}
\label{lemma:Matrix_follow}
The matrices $\bbB\in \reals^{d\times d}$, $\bbQ_x = (\bbI + \eta^2 \bbB \bbB^{\top})^{-1}$, and $\bbQ_y = (\bbI + \eta^2 \bbB^{\top} \bbB)^{-1}$ satisfy the following properties:
\begin{align}
\bbQ_x \bbB &= \bbB\bbQ_y , \label{eq:follow_1} \\
\bbQ_y \bbB^{\top} & = \bbB^{\top} \bbQ_x. \label{eq:follow_2}
\end{align} 
\end{lemma}
\begin{proof}
Let $\bbB = \bbU \bbLambda \bbV^{\top}$ be the singular value decomposition of $\bbB$. Here $\bbU$ and $\bbV$ are orthonormal matrices and $\bbLambda$ is a diagonal matrix with the eigenvalues of $\bbB$ as the diagonal entries. Then, we have:
\begin{align}
\bbQ_x \bbB &= (\bbI + \eta^2 \bbU \bbLambda \bbV^{\top} \bbV \bbLambda \bbU^{\top})^{-1} \bbU \bbLambda \bbV^{\top} \nonumber \\
&= (\bbU (\eta^2 \bbLambda^2 + \bbI) \bbU^{\top})^{-1}  \bbU \bbLambda \bbV^{\top} \nonumber \\
&= \bbU (\eta^2\bbLambda^2 + \bbI)^{-1} \bbU^{\top}  \bbU \bbLambda \bbV^{\top} \nonumber \\
&= \bbU (\eta^2\bbLambda^2 + \bbI)^{-1} \bbLambda \bbV^{\top}
\end{align}
Here we used the property that $\bbU^{\top} \bbU = \bbV^{\top} \bbV = \bbI$. Now, we simplify the other side to get:
\begin{align}
\bbB\bbQ_y &= \bbU \bbLambda \bbV^{\top}  (\bbI + \eta^2 \bbV \bbLambda \bbU^{\top} \bbU \bbLambda \bbV^{\top} )^{-1} \nonumber \\
&=  \bbU \bbLambda \bbV^{\top} (\bbV (\eta^2 \bbLambda^2 + \bbI) \bbV^{\top})^{-1} \nonumber \\
&= \bbU \bbLambda \bbV^{\top}  \bbV (\eta^2\bbLambda^2 + \bbI)^{-1} \bbV^{\top} \nonumber \\
&= \bbU \bbLambda (\eta^2\bbLambda^2 + \bbI)^{-1}  \bbV^{\top}
\end{align}
Now, since $ \bbU (\eta^2\bbLambda^2 + \bbI)^{-1} \bbLambda^2 \bbV^{\top} =  \bbU \bbLambda (\eta^2\bbLambda + \bbI)^{-1}  \bbV^{\top}$, the claim in \eqref{eq:follow_1}   follows. Using a similar argument we can also prove the equality in \eqref{eq:follow_2}. 
\end{proof}

Using the result in Lemma~\ref{lemma:Matrix_follow} we can show that 
\begin{align}
\bbx_k^\top \bbQ_x \bbB\bby_k=\bbx_k^\top \bbB\bbQ_y\bby_k=
\bby_k^\top \bbQ_y \bbB^\top\bbx_k,
\end{align}
where the second equality holds as $\bba^\top\bbb=\bbb^\top\bba$. Hence, the expression in \eqref{eq:prox_bilinear_1} can be simplified as
\begin{align}\label{eq:prox_bilinear_1} 
\|\bbx_{k+1}\|^2 + \|\bby_{k+1}\|^2 &= \|\bbQ_x \bbx_k\|^2 + \eta^2\|\bbQ_x \bbB \bby_k\|^2 + \|\bbQ_y \bby_k\|^2 + \eta^2 \|\bbQ_y \bbB^{\top} \bbx_k\|^2.
\end{align}
We simplify equation  \eqref{eq:prox_bilinear_1} as follows. Consider the term involving $\bbx_k$. We have
\begin{align}\label{eq:bi_pox_proof_400}
\|\bbQ_x \bbx_k\|^2 +\eta^2 \|\bbQ_y \bbB^{\top} \bbx_k\|^2 &= \bbx_k^{\top} \bbQ_x^2 \bbx_k + \eta^2 \bbx_k^{\top} \bbB \bbQ_y^2 \bbB^{\top} \bbx_k \nonumber \\
&= \bbx_k^{\top} (\bbQ_x^2 + \eta^2 \bbB \bbQ_y^2 \bbB^{\top}) \bbx_k  
\end{align}
Now we use Lemma \ref{lemma:Matrix_follow} to simplify \eqref{eq:bi_pox_proof_400} as follows
\begin{align}\label{eq:bi_pox_proof_600}
\|\bbQ_x \bbx_k\|^2 +\eta^2 \|\bbQ_y \bbB^{\top} \bbx_k\|^2 
&=  \bbx_k^{\top} (\bbQ_x^2 + \eta^2 \bbB \bbQ_y^2 \bbB^{\top}) \bbx_k \nonumber  \\
&=  \bbx_k^{\top} (\bbQ_x^2 + \eta^2 \bbB \bbQ_y \bbB^{\top}\bbQ_x) \bbx_k \nonumber  \\
&=  \bbx_k^{\top} (\bbQ_x^2 + \eta^2 \bbB  \bbB^{\top}\bbQ_x\bbQ_x) \bbx_k \nonumber  \\
&=  \bbx_k^{\top} (\bbI + \eta^2 \bbB  \bbB^{\top})\bbQ_x^2 \bbx_k \nonumber  \\
&= \bbx_k^{\top} (\bbI + \eta^2 \bbB \bbB^{\top})^{-1} \bbx_k,
\end{align}
where the last equality follows by replacing $\bbQ_x$ by its definition. The same simplification follows for the terms invovling $\bby_k$ which leads to the expression 
\begin{align}\label{eq:bi_pox_proof_700}
 \|\bbQ_y \bby_k\|^2 + \eta^2\|\bbQ_x \bbB \bby_k\|^2 
 =\bby_k^{\top} (\bbI + \eta^2 \bbB^{\top} \bbB)^{-1} \bby_k.
\end{align}
Substitute $\|\bbQ_x \bbx_k\|^2 +\eta^2 \|\bbQ_y \bbB^{\top} \bbx_k\|^2 $ and $ \|\bbQ_y \bby_k\|^2 + \eta^2\|\bbQ_x \bbB \bby_k\|^2 $ in \eqref{eq:prox_bilinear_1}  with the expressions in \eqref{eq:bi_pox_proof_600} and \eqref{eq:bi_pox_proof_700}, respectively, to obtain
\begin{align}\label{eq:prox_bilinear_800} 
\|\bbx_{k+1}\|^2 + \|\bby_{k+1}\|^2 &= \bbx_k^{\top} (\bbI + \eta^2 \bbB \bbB^{\top})^{-1} \bbx_k+\bby_k^{\top} (\bbI + \eta^2 \bbB^{\top} \bbB)^{-1} \bby_k.
\end{align}

 Now, using the expression in \eqref{eq:prox_bilinear_800}  and the fact that $\lambda_{\min}(\bbB^T \bbB) = \lambda_{\min}(\bbB \bbB^T)$ we can write
\begin{align}
\|\bbx_{k+1}\|^2 + \|\bby_{k+1}\|^2 
\leq \bigg( \frac{1}{1 + \eta^2\lambda_{\min}(\bbB^{\top} \bbB)} \bigg) (\|\bbx_k\|^2 + \|\bby_k\|^2) \label{eq:prox_bilinear_2},
\end{align}
and the claim in Theorem \ref{thm:bilinear_ppm} follows. 

\subsection{Proof of Theorem~\ref{thm:scvx_sccv_ppm}}
The update of PP method can be written as 
\begin{align}
\bbx_{k+1} &= \bbx_{k} - \eta \nabla_{\bbx} f (\bbx_{k+1},  \bby_{k+1}), \nonumber \\
\bby_{k+1} &= \bby_k + \eta \nabla_{\bby} f (\bbx_{k+1},  \bby_{k+1}). 
\end{align}
Consider the function $\phi_f:\reals^m \to \reals$ defined as
\begin{align}\label{phi_def}
\phi_{\bby_{k+1}}(\bbx) := f(\bbx,\bby_{k+1}) + \frac{1}{2\eta} \| \bbx - \bbx_k\|^2.
\end{align}
It is easy to check that $\phi_f$ is $\mu_x+ \frac{1}{\eta}$ strongly convex, and it also can be verified that $\bbx_{k+1}=\argmin_{\bbx}\phi_f(\bbx)$. Hence, using strong convexity of $\phi_f$, for any $\bbx\in \reals^m$, we have
\begin{align}
\phi_{\bby_{k+1}}(\bbx) - \phi_{\bby_{k+1}}(\bbx_{k+1}) \geq \frac{1}{2} \bigg( \mu_x + \frac{1}{\eta} \bigg) \|\bbx - \bbx_{k+1} \| ^2,
\end{align}
where we used the fact that $\nabla \phi_{\bby_{k+1}}(\bbx_{k+1}) = \bb0$.
Replace $\phi_{\bby_{k+1}}(\bbx)$ and $ \phi_{\bby_{k+1}}(\bbx_{k+1})$ with their definition in \eqref{phi_def} and further set $\bbx = \bbx^*$ to obtain
\begin{align}
 f(\bbx^*,  \bby_{k+1})  -  f(\bbx_{k+1}, \bby_{k+1})
 &\geq  \frac{1}{2} \bigg( \mu_x + \frac{1}{\eta} \bigg) \|\bbx_{k+1} - \bbx^* \| ^2 - \frac{1}{2\eta} \| \bbx_k - \bbx^*\|^2 + \frac{1}{2\eta} \|\bbx_{k+1} - \bbx_k \|^2 \nonumber \\
&\geq \frac{1}{2} \bigg( \mu_x + \frac{1}{\eta} \bigg) \|\bbx_{k+1} - \bbx^* \| ^2 - \frac{1}{2\eta} \| \bbx_k - \bbx^*\|^2.
\label{eq:Strong_Prox_Primal}
\end{align}
Once again, consider the function:
\begin{align}
\phi_{\bbx_{k+1}} (\bby) = - f(\bbx_{k+1},\bby) + \frac{1}{2\eta} \| \bby - \bby_k\|^2.
\end{align}
It is $\mu_y + \frac{1}{\eta}$ strongly convex and is minimized at $\bby_{k+1}$. Therefore, for any $\bby\in \reals^n$, we have
\begin{align}
\phi_{\bbx_{k+1}} (\bby) - \phi_{\bbx_{k+1}} (\bby_{k+1}) \geq \frac{1}{2} \bigg( \mu_y + \frac{1}{\eta} \bigg) \|\bby - \bby_{k+1} \| ^2
\end{align}
since $\nabla \phi_{\bbx_{k+1}} (\bby_{k+1}) = \bb0$. Replace $\phi_{\bbx_{k+1}} (\bby)$ and $ \phi_{\bbx_{k+1}} (\bby_{k+1})$ with their definitions and further set $\bby = \bby^*$ to obtain
\begin{align}
 f(\bbx_{k+1}, \bby_{k+1}) - f(\bbx_{k+1},  \bby^*)
&\geq  \frac{1}{2} \bigg( \mu_y + \frac{1}{\eta} \bigg) \|\bby_{k+1} - \bby^* \| ^2 - \frac{1}{2\eta} \| \bby_k - \bby^*\|^2 + \frac{1}{2\eta} \|\bby_{k+1} - \bby_k \|^2 \nonumber \\
&\geq \frac{1}{2} \bigg( \mu_y + \frac{1}{\eta} \bigg) \|\bby_{k+1} - \bby^* \| ^2 - \frac{1}{2\eta} \| \bby_k - \bby^*\|^2.
\label{eq:Strong_Prox_Dual}
\end{align}
The saddle point property implies that the optimal solution set $(\bbx^*,\bby^*)$ satisfies the following inequalities for any $\bbx\in \reals^m$ and $ \bby\in \reals^n$:
\begin{align}
f(\bbx^*, \bby) \leq f(\bbx^*, \bby^*) \leq f(\bbx, \bby^*) .
\end{align}
In particular, by setting $(\bbx, \bby) = (\bbx_{k+1}, \bby_{k+1})$ we obtain that
\begin{align}\label{saddle_prop}
f(\bbx^*, \bby_{k+1}) \leq f(\bbx^*, \bby^*) \leq f(\bbx_{k+1}, \bby^*).
\end{align}
Now, considering \eqref{eq:Strong_Prox_Primal}, by adding and subtracting $f(\bbx^*, \bby^*)$ we can write
\begin{align}
&f(\bbx^*,  \bby_{k+1}) -  f(\bbx^*, \bby^*) + f(\bbx^*, \bby^*) - f(\bbx_{k+1}, \bby_{k+1}) 
\nonumber\\
&\quad \geq \frac{1}{2} \bigg( \mu_x + \frac{1}{\eta} \bigg) \|\bbx_{k+1} - \bbx^* \| ^2 - \frac{1}{2\eta} \| \bbx_k - \bbx^*\|^2,
\end{align}
Regroup the terms to obtain
\begin{align}
& f(\bbx^*,  \bby_{k+1}) -  f(\bbx^*, \bby^*) \nonumber\\
&\quad \geq \frac{1}{2} \bigg( \mu_x + \frac{1}{\eta} \bigg) \|\bbx_{k+1} - \bbx^* \| ^2 - \frac{1}{2\eta} \| \bbx_k - \bbx^*\|^2 - f(\bbx^*, \bby^*) + f(\bbx_{k+1}, \bby_{k+1}). 
\end{align}
By using the inequality in \eqref{saddle_prop} we can write
\begin{align}
 \frac{1}{2} \bigg( \mu_x + \frac{1}{\eta} \bigg) \|\bbx_{k+1} - \bbx^* \| ^2 - \frac{1}{2\eta} \| \bbx_k - \bbx^*\|^2 - f(\bbx^*, \bby^*) + f(\bbx_{k+1}, \bby_{k+1}) \leq 0
\label{eq:Strong_Prox_1}
\end{align}

\noindent Similarly, considering \eqref{eq:Strong_Prox_Dual}, we can write
\begin{align}
&f(\bbx_{k+1}, \bby_{k+1}) -f(\bbx^*, \bby^*)  +f(\bbx^*, \bby^*)   - f(\bbx_{k+1},  \bby^*) \nonumber\\
&\quad  \geq \frac{1}{2} \bigg( \mu_y + \frac{1}{\eta} \bigg) \|\bby_{k+1} - \bby^* \| ^2 - \frac{1}{2\eta} \| \bby_k - \bby^*\|^2, 
\end{align}
and, therefore,
\begin{align}
\frac{1}{2} \bigg( \mu_y + \frac{1}{\eta} \bigg) \|\bby_{k+1} - \bby^* \| ^2 - \frac{1}{2\eta} \| \bby_k - \bby^*\|^2 - f(\bbx_{k+1}, \bby_{k+1}) + f(\bbx^*, \bby^*) \leq 0.
\label{eq:Strong_Prox_2}
\end{align}
Add equations \eqref{eq:Strong_Prox_1} and \eqref{eq:Strong_Prox_2}, and use the definition $\mu = \min \{ \mu_x,\mu_y \} $ to obtain
\begin{align}
\frac{1}{2} \bigg( \mu + \frac{1}{\eta} \bigg) \bigg( \|\bbx_{k+1} - \bbx^* \| ^2  + \|\bby_{k+1} - \bby^* \| ^2 \bigg) \leq \frac{1}{2\eta} \bigg( \|\bbx_{k} - \bbx^* \| ^2  + \|\bby_{k} - \bby^* \| ^2 \bigg) .
\end{align}
Regrouping the terms and using the definition $r_k = \|\bbx_{k} - \bbx^* \| ^2  + \|\bby_{k} - \bby^* \| ^2$ leads to
\begin{align}
r_{k+1} &\leq \frac{1}{\eta} \bigg( \mu + \frac{1}{\eta} \bigg)^{-1} r_k \nonumber \\
&= \frac{1}{1 + \eta \mu} r_k,
\end{align}
and the proof is complete.

\subsection{Proof of Proposition~\ref{prop:OGDA_Approx}}

We start from the Proximal Point (PP) dynamics and show that an $\mathcal{O}(\eta^2)$ approximation of this dynamics leads to OGDA. The PP updates are as follows
 \begin{align}
 \bbx_{k+1} &= \bbx_k - \eta \nabla_{\bbx} f \big( \bbx_k - \eta \nabla_{\bbx} f(\bbx_{k+1}, \bby_{k+1}),\bby_k + \eta \nabla_{\bby} f(\bbx_{k+1},\bby_{k+1}) \big) \nonumber \\
  \bbx_{k+1} &= \bbx_k - \eta \nabla_{\bbx} f \big( \bbx_k - \eta \nabla_{\bbx} f(\bbx_{k+1}, \bby_{k+1}),\bby_k + \eta \nabla_{\bby} f(\bbx_{k+1},\bby_{k+1}) \big) \nonumber 
\end{align}

By writing the Taylor's expansion of $\nabla_{\bbx} f$, we obtain
\begin{align}
& \nabla_{\bbx} f \left(\bbx_{k} - \eta \nabla_{\bbx} f(\bbx_{k+1}, \bby_{k+1}), \bby_k + \eta \nabla_{\bby} f(\bbx_{k+1}, \bby_{k+1})\right)\nonumber\\
& =  \nabla_{\bbx} f(\bbx_k ,\bby_k) 
+   \nabla_{\bbx\bbx} f(\bbx_k ,\bby_k) [\bbx_k - \eta \nabla_{\bbx} f(\bbx_{k+1}, \bby_{k+1})-\bbx_k] \nonumber \\
& \qquad \qquad \qquad \qquad \qquad \qquad +   \nabla_{\bbx\bby} f(\bbx_k ,\bby_k) [\bby_k + \eta \nabla_{\bby} f(\bbx_{k+1}, \bby_{k+1})-\bby_k] +o(\eta)\nonumber\\
& =  \nabla_{\bbx} f(\bbx_k ,\bby_k) 
- \eta  \nabla_{\bbx\bbx} f(\bbx_{k} ,\bby_k) \nabla_{\bbx} f(\bbx_{k+1}, \bby_{k+1})
\nonumber \\
& \qquad \qquad \qquad \qquad \qquad \qquad+ \eta  \nabla_{\bbx\bby} f(\bbx_k ,\bby_k) \nabla_{\bby} f(\bbx_{k+1}, \bby_{k+1})+o(\eta).
\end{align}

Using this expression, we have
\begin{align}
\bbx_{k+1} &= \bbx_k  - \eta \nabla_{\bbx} f(\bbx_k ,\bby_k) + \eta^2  \nabla_{\bbx\bbx} f(\bbx_{k} ,\bby_k)  \nabla_{\bbx} f(\bbx_{k+1}, \bby_{k+1}) \nonumber \\
&\quad - \eta^2  \nabla_{\bbx\bby} f(\bbx_k ,\bby_k) \nabla_{\bby} f(\bbx_{k+1}, \bby_{k+1})+o(\eta^2)
\end{align}
On adding and subtracting the term $\eta \nabla_{\bbx} f(\bbx_k ,\bby_k)$, we get
\begin{align}
\bbx_{k+1}&=  \bbx_k - 2\eta \nabla_{\bbx} f(\bbx_k,\bby_k)  + \eta\big( \eta \nabla_{\bbx \bbx} f(\bbx_k,\bby_k) \nabla_{\bbx} f(\bbx_{k+1},\bby_{k+1}) + \nabla_{\bbx} f(\bbx_k,\bby_k)  \nonumber \\
&  \quad   - \eta \nabla_{\bbx \bby}f(\bbx_k,\bby_k) \nabla_{\bby} f(\bbx_{k+1},\bby_{k+1})\big) + o(\eta^2)  \label{eq:OGD_partial}
\end{align}
Note that from the Taylors expansion of $ \nabla_{\bbx \bbx} f, \nabla_{\bbx} f$ and the PP updates, we have
\begin{align}\label{def_errors}
\nabla_{\bbx \bbx} f(\bbx_k,\bby_k) = \nabla_{\bbx \bbx} f(\bbx_{k-1},\bby_{k-1}) + \mathcal{O}(\eta) ,\ \ 
\nabla_{\bbx} f(\bbx_{k+1},\bby_{k+1}) = \nabla_{\bbx} f(\bbx_{k},\bby_{k}) + \mathcal{O}(\eta)
\end{align}
which leads to
\begin{align}
\label{eq:OGD_approx_1}
\eta \nabla_{\bbx \bbx} f(\bbx_k,\bby_k) \nabla_{\bbx} f(\bbx_{k+1},\bby_{k+1}) &= \eta \nabla_{\bbx \bbx} f(\bbx_{k-1},\bby_{k-1}) \nabla_{\bbx} f(\bbx_{k},\bby_{k}) + \mathcal{O}(\eta^2) 
\end{align}
Again, from the Taylor's expansion of $ \nabla_{\bbx \bby} f, \nabla_{\bby} f$ and the PP updates, we have
\begin{align}\label{def_errors}
\nabla_{\bbx \bby} f(\bbx_k,\bby_k) = \nabla_{\bbx \bby} f(\bbx_{k-1},\bby_{k-1}) + \mathcal{O}(\eta) ,\ \ 
\nabla_{\bby} f(\bbx_{k+1},\bby_{k+1}) = \nabla_{\bby} f(\bbx_{k},\bby_{k}) + \mathcal{O}(\eta)
\end{align}
which implies that
\begin{align}
\eta  \nabla_{\bbx \bby} f(\bbx_k,\bby_k) \nabla_{\bby} f(\bbx_{k+1},\bby_{k+1}) &= \eta \nabla_{\bbx \bby} f(\bbx_{k-1},\bby_{k-1}) \nabla_{\bby} f(\bbx_{k},\bby_{k}) + \mathcal{O}(\eta^2) 
\label{eq:OGD_approx_2}
\end{align}
Making the approximations of Equations \eqref{eq:OGD_approx_1} and \eqref{eq:OGD_approx_2} in Equation \eqref{eq:OGD_partial} yields
\begin{align}
\bbx_{k+1} &= \bbx_k - 2\eta \nabla_{\bbx} f(\bbx_k,\bby_k) + \eta\big( \eta \nabla_{\bbx \bbx} f(\bbx_{k-1},\bby_{k-1}) \nabla_{\bbx} f(\bbx_{k},\bby_{k})  + \nabla_{\bbx} f(\bbx_k,\bby_k)  \nonumber \\
&\quad \quad \quad - \eta \nabla_{\bbx \bby}f(\bbx_{k-1},\bby_{k-1}) \nabla_{\bby} f(\bbx_{k},\bby_{k}) + \mathcal{O}(\eta^2) \big) + o(\eta^2) 
\label{eq:OGD_Partial_2}
\end{align}
We also know that
\begin{align}
\nabla_{\bbx}   f(\bbx_{k-1}, \bby_{k-1}) + \mathcal{O}(\eta^2) &= \eta \nabla_{\bbx \bbx} f(\bbx_{k-1},\bby_{k-1})  \nabla_{\bbx} f(\bbx_{k},\bby_{k}) + \nabla_{\bbx} f(\bbx_k,\bby_k)  \nonumber \\
&\quad - \eta \nabla_{\bbx \bby}f(\bbx_{k-1},\bby_{k-1}) \nabla_{\bby} f(\bbx_{k},\bby_{k}) + \mathcal{O}(\eta^2) \nonumber 
\end{align}
Making this substitution back in Equation \eqref{eq:OGD_Partial_2}, we get
\begin{align}
\bbx_{k+1} &= \bbx_k - 2\eta \nabla_{\bbx} f(\bbx_k,\bby_k) + \eta\big(\nabla_{\bbx}  f(\bbx_{k-1}, \bby_{k-1})  + \mathcal{O}(\eta^2)\big) + o(\eta^2) \nonumber \\
&= \bbx_k - 2\eta \nabla_{\bbx} f(\bbx_k,\bby_k)  + \eta\nabla_{\bbx}  f(\bbx_{k-1},\bby_{k-1}) + o(\eta^2)
\end{align}
which is equivalent to the OGDA update plus an additional error term of order $o(\eta^2)$. The same analysis can be done for the dual updates as well to obtain
\begin{align}
\bby_{k+1} 
&= \bby_k + 2\eta \nabla_{\bby} f(\bbx_k,\bby_k)  - \eta\nabla_{\bbx}  f(\bbx_{k-1},\bby_{k-1}) + o(\eta^2).
\end{align}
This shows that the OGDA updates and the PP updates differ by $o(\eta^2)$.

\subsection{Proof of Theorem~\ref{thm:bili}}

We define the following symmetric matrices
\begin{align}
\bbE_x &= \bbI - \eta^2  \bbB \bbB^{\top} - (\bbI + \eta^2 \bbB \bbB^{\top})^{-1}, \nonumber \\
\bbE_y &= \bbI - \eta^2  \bbB^{\top}  \bbB - (\bbI + \eta^2 \bbB^{\top}  \bbB)^{-1}.\nonumber
\end{align}
We rewrite the properties of $\bbE_x$ and  $\bbE_y$ which are
\begin{align}
\|\bbE_x\|, \|\bbE_y\| &\leq  \frac{\eta^4 \lambda_{\max}(\bbB^{\top} \bbB)^2}{1 - \eta^2 \sqrt{\lambda^2_{\max}(\bbB^{\top} \bbB)}} =e \label{eq:equa_error_1}
\end{align}
\begin{align}
\bbE_x \bbB &= \bbB \bbE_y \label{eq:error_flow_11} \\
\bbE_y \bbB^{\top} &= \bbB^{\top} \bbE_x \label{eq:error_flow_22} 
\end{align}

Recall that the update of OGDA for the bilinear problem can be written as
\begin{align}
\bbx_{k+1} &=  \bbx_k - 2 \eta \bbB \bby_k  + \eta \bbB \bby_{k-1}, \nonumber \\
\bby_{k+1} &=  \bbx_k + 2 \eta \bbB^{\top} \bbx_k  + \eta \bbB^{\top} \bbx_{k-1} .\nonumber 
\end{align}
The update for the variable $\bbx$ can be written as an approximate variant of the PP update as follows
\begin{align}
\bbx_{k+1} 
&=(\bbI+\eta^2\bbB\bbB^\top)^{-1} (\bbx_k-\eta\bbB\bby_k) 
\nonumber\\
&\quad -\left[\big( \bbx_k - 2\eta \bbB \bby_k + \eta \bbB \bby_{k-1} \big)  - \big( (\bbI + \eta^2 \bbB \bbB^{\top})^{-1} (\bbx_k - \eta \bbB y_k) \big)\right]\nonumber\\
&= (\bbI+\eta^2\bbB\bbB^\top)^{-1} (\bbx_k-\eta\bbB\bby_k) \nonumber \\
& \quad  -\left[(-\eta \bbB \bby_k + \eta \bbB \bby_{k-1} + \eta^2 \bbB \bbB^{\top} \bbx_k - \eta^3 \bbB \bbB^{\top} \bbB \bby_k) + \bbE_x(\bbx_k - \eta \bbB \bby_k)\right]
\end{align}
Therefore, the error between the OGDA and Proximal updates for the variable $\bbx$ is given by
\begin{align}\label{jafang}
 (-\eta \bbB \bby_k + \eta \bbB \bby_{k-1} + \eta^2 \bbB \bbB^{\top} \bbx_k - \eta^3 \bbB \bbB^{\top} \bbB \bby_k) + \bbE_x(\bbx_k - \eta \bbB \bby_k)
\end{align}
We first derive an upper bound for the term in the first parentheses $(-\eta \bbB \bby_k + \eta \bbB \bby_{k-1} + \eta^2 \bbB \bbB^{\top} \bbx_k - \eta^3 \bbB \bbB^{\top} \bbB \bby_k) $.

Using the OGDA update, we have:
\begin{align}
-\eta \bbB \bby_k + \eta \bbB \bby_{k-1} &= -\eta \bbB (\bby_k - \bby_{k-1} ) \nonumber \\
&= -\eta \bbB ( 2\eta \bbB^{\top} \bbx_{k-1} - \eta \bbB^{\top} \bbx_{k-2} )
\end{align}
Therefore, we can write
\begin{align}
& (-\eta \bbB \bby_k + \eta \bbB \bby_{k-1}  + \eta^2 \bbB \bbB^{\top} \bbx_k - \eta^3 \bbB \bbB^{\top} \bbB \bby_k) \nonumber \\
 &\quad = (-2\eta^2 \bbB \bbB^{\top} \bbx_{k-1} + \eta^2 \bbB \bbB^{\top} \bbx_{k-2} + \eta^2 \bbB \bbB^{\top} \bbx_k - \eta^3 \bbB \bbB^{\top} \bbB \bby_k ) \nonumber \\
&\quad = (\eta^2\bbB \bbB^{\top} \bbx_k  - \eta^2 \bbB \bbB^{\top} \bbx_{k-1}  - \eta^2 \bbB \bbB^{\top} \bbx_{k-1} + \eta^2 \bbB \bbB^{\top} \bbx_{k-2} - \eta^3 \bbB \bbB^{\top} \bbB \bby_k ) \nonumber 
\end{align}
Once again, using the OGDA updates for $(\bbx_k - \bbx_{k-1})$ and $(\bbx_{k-1} - \bbx_{k-2})$, we have
\begin{align}\label{jafang22}
&(\eta^2\bbB \bbB^{\top} \bbx_k   - \eta^2 \bbB \bbB^{\top} \bbx_{k-1}  - \eta^2 \bbB \bbB^{\top} \bbx_{k-1} + \eta^2 \bbB \bbB^{\top} \bbx_{k-2} - \eta^3 \bbB \bbB^{\top} \bbB \bby_k ) \nonumber \\
&\quad = (-2\eta^3 \bbB \bbB^{\top} \bbB \bby_{k-1} + 3\eta^3 \bbB \bbB^{\top} \bbB \bby_{k-2} - \eta^3\bbB \bbB^{\top} \bbB \bby_{k-3} - \eta^3\bbB \bbB^{\top} \bbB \bby_k) \nonumber \\
&\quad = -\eta^3 \bbB \bbB^{\top} \bbB ( \bby_k + 2\bby_{k-1} - 3\bby_{k-2} + \bby_{k-1} )
\end{align}
Therefore, considering the expressions in \eqref{jafang} and \eqref{jafang22} the error between the updates of OGDA and PP for the variable $\bbx$ can be written as
\begin{align}
&\big( \bbx_k - 2\eta \bbB \bby_k + \eta \bbB \bby_{k-1} \big)  -  \big( (\bbI + \eta^2 \bbB \bbB^{\top})^{-1} (\bbx_k - \eta \bbB y_k) \big) \nonumber \\
&\quad =  \bbE_x(\bbx_k - \eta \bbB \bby_k) -\eta^3 \bbB \bbB^{\top} \bbB ( \bby_k + 2\bby_{k-1} - 3\bby_{k-2} + \bby_{k-3} )
\end{align}
We apply the same argument for the update of the variable $\bby$. Combining these results we obtain that the update of OGDA can be written as
\begin{align}
\bbx_{k+1} &= (\bbI  + \eta^2 \bbB \bbB^{\top})^{-1} (\bbx_k - \eta \bbB \bby_k) + \bbE_x(\bbx_k - \eta \bbB \bby_k) 
\nonumber\\
&\quad -\eta^3 \bbB \bbB^{\top} \bbB ( \bby_k + 2\bby_{k-1} - 3\bby_{k-2} + \bby_{k-3} ) \nonumber \\
\bby_{k+1} &= (\bbI  + \eta^2 \bbB^{\top} \bbB)^{-1} (\bby_k + \eta \bbB^{\top} \bbx_k) + \bbE_y(\bby_k + \eta \bbB^{\top} \bbx_k) \nonumber\\
&\quad +\eta^3 \bbB^{\top} \bbB \bbB^{\top} ( \bbx_k + 2\bbx_{k-1} - 3\bbx_{k-2} + \bbx_{k-3} )
\end{align}
As in the proof of Theorem \ref{thm:bilinear_ppm}, we define $\bbQ_x = (\bbI + \eta^2 \bbB \bbB^{\top})^{-1}$ and $\bbQ_y = (\bbI + \eta^2 \bbB^{\top} \bbB)^{-1}$. Then, we can show that 
\begin{align}
\|\bbx_{k+1}\|^2 &\leq (\bbx_k - \eta \bbB \bby_k)^{\top} \bbQ_x^2 (\bbx_k - \eta \bbB \bby_k)  \nonumber \\
& \qquad +  \|\bbE_x(\bbx_k - \eta \bbB \bby_k) 
- \eta^3 \bbB \bbB^{\top} \bbB ( \bby_k + 2\bby_{k-1} - 3\bby_{k-2} + \bby_{k-3} )\|^2 \nonumber \\
& \qquad + 2(\bbx_k - \eta \bbB \bby_k)^{\top} \bbQ_x \big(  \bbE_x(\bbx_k - \eta \bbB \bby_k) 
\nonumber \\
& \qquad -\eta^3 \bbB \bbB^{\top} \bbB ( \bby_k + 2\bby_{k-1} - 3\bby_{k-2} + \bby_{k-3} ) \big) \nonumber \\
\|\bby_{k+1}\|^2 &\leq (\bby_k + \eta \bbB^{\top} \bbx_k)^{\top} \bbQ_y^2 (\bby_k + \eta \bbB^{\top} \bbx_k)  \nonumber \\
& \qquad +  \| \bbE_y (\bby_k + \eta \bbB^{\top} \bbx_k) + \eta^3  \bbB^{\top} \bbB \bbB^{\top}( \bbx_k + 2\bbx_{k-1} - 3\bbx_{k-2} + \bbx_{k-3} )\|^2 \nonumber \\
& \qquad + 2(\bby_k + \eta \bbB^{\top} \bbx_k)^{\top} \bbQ_y \big(  \bbE_y(\bby_k + \eta \bbB^{\top} \bbx_k) 
\nonumber \\
& \qquad + \eta^3  \bbB^{\top} \bbB \bbB^{\top} ( \bbx_k + 2\bbx_{k-1} - 3\bbx_{k-2} + \bbx_{k-3} ) \big)
\end{align}

On summing the two sides, we have:
\begin{align}
&\|\bbx_{k+1}\|^2 + \|\bby_{k+1}\|^2 \nonumber\\
&\leq (\bbx_k - \eta \bbB \bby_k)^{\top} \bbQ_x^2 (\bbx_k - \eta \bbB \bby_k) +  (\bby_k + \eta \bbB^{\top} \bbx_k)^{\top} \bbQ_y^2 (\bby_k + \eta \bbB^{\top} \bbx_k)  \nonumber \\
& +  \|\bbE_x(\bbx_k - \eta \bbB \bby_k) - \eta^3 \bbB \bbB^{\top} \bbB ( \bby_k + 2\bby_{k-1} - 3\bby_{k-2} + \bby_{k-3} )\|^2 \nonumber \\
&+  \| \bbE_y (\bby_k + \eta \bbB^{\top} \bbx_k) + \eta^3  \bbB^{\top} \bbB \bbB^{\top}( \bbx_k + 2\bbx_{k-1} - 3\bbx_{k-2} + \bbx_{k-3} )\|^2 \nonumber \\
& + 2(\bbx_k - \eta \bbB \bby_k)^{\top} \bbQ_x \big(  \bbE_x(\bbx_k - \eta \bbB \bby_k) -\eta^3 \bbB \bbB^{\top} \bbB ( \bby_k + 2\bby_{k-1} - 3\bby_{k-2} + \bby_{k-3} ) \big) \nonumber \\
& + 2(\bby_k + \eta \bbB^{\top} \bbx_k)^{\top} \bbQ_y \big(  \bbE_y(\bby_k + \eta \bbB^{\top} \bbx_k) + \eta^3  \bbB^{\top} \bbB \bbB^{\top} ( \bbx_k + 2\bbx_{k-1} - 3\bbx_{k-2} + \bbx_{k-3} ) \big)
\end{align}
Define $r_k = \|\bbx_k\|^2 + \|\bby_k\|^2$. We have:
\begin{align}
r_{k+1} \leq  \max_{i \in \{ k, k-1, k-2, k-3\} } &\Bigg[ \bbx_i^{\top} (\bbQ_x + 2\bbE_x^2 + 2\eta^2 \bbB \bbB^{\top} \bbE_x^2 + 30\eta^6 (\bbB \bbB^{\top})^3 
+  2\bbQ_x \bbE_x 
\nonumber\\
&\quad  + 2\eta^2 \bbB \bbQ_y \bbE_y \bbB^{\top} - 20\eta^3 \bbB^{\top} \bbB \bbB^{\top} \bbQ_x ) \bbx_i + \bby_i^{\top} (\bbQ_y + 2\bbE_y^2 
\nonumber \\
&\quad + 2\eta^2 \bbB \bbB^{\top} \bbE_y^2 + 30\eta^6 (\bbB \bbB^{\top})^3 +  2\bbQ_y \bbE_y  + 2\eta^2 \bbB \bbQ_x \bbE_x \bbB^{\top} 
\nonumber \\
&\quad- 20\eta^3 \bbB^{\top} \bbB \bbB^{\top} \bbQ_y ) \bby_i \Bigg]
\end{align}
And for $\eta = \frac{1}{40\sqrt{\lambda_{\max} (\bbB^{\top} \bbB)} }$. 
\begin{align}
r_{k+1} &\leq \!\!\max_{i \in \{ k, k-1, k-2, k-3\} } \!\left[ \bbx_i^{\top} (\bbI - \frac{1}{2}\eta^2 \bbB \bbB^{\top}\!\! +\!  \frac{1}{4}\eta^2 \bbB \bbB^{\top}) \bbx_i + \bby_i^{\top} (\bbI - \frac{1}{2}\eta^2 \bbB^{\top} \bbB +  \frac{1}{4}\eta^2 \bbB^{\top} \bbB) \bby_i \right] \nonumber \\
&\leq \left( 1 - \frac{1}{800\kappa} \right) \max \{ r_{k}, r_{k-1}, r_{k-2}, r_{k-3} \}
\end{align}

\subsection{Proof of Theorem~\ref{thm:scv}}

We define $\bbz = [\bbx; \bby]$ and $F(\bbz) = [\nabla_{\bbx} f(\bbx, \bby) ; -\nabla_{\bby} f(\bbx, \bby)]$. Then the OGDA updates can be compactly written as:
\begin{align}
\bbz_{k+1} = \bbz_k - 2\eta F(\bbz_{k}) + \eta F(\bbz_{k-1})
\end{align}
We write the update in terms of the Proximal Point method with an error $\bbvarepsilon_k = \eta(F(\bbz_{k+1}) - 2F(\bbz_{k}) + F(\bbz_{k-1}))$ as follows:
\begin{align}
\bbz_{k+1} = \bbz_k - \eta F(\bbz_{k+1}) + \bbvarepsilon_k
\label{eq:PP_with_error}
\end{align}
We have:
\begin{align}
\| \bbz_{k+1} - \bbz^* \|^2 = \| \bbz_k - \bbz^*\|^2 - \| \bbz_{k+1} - \bbz_{k} \|^2 - 2\eta (F(\bbz_{k+1}) + \bbvarepsilon_k)^{\top} (\bbz_{k+1} - \bbz^*)
\label{eq:OGDA_Main_1}
\end{align}
On rearranging Equation \eqref{eq:PP_with_error} and using the fact that $F(\bbz^*) = 0$, where $\bbz^* = [\bbx^*; \bby^*]$, we get:
\begin{align}
\eta (F(\bbz_{k+1}) - F(\bbz^*)) = \bbz_k - \bbz_{k+1} + \eta(F(\bbz_{k+1}) - F(\bbz_{k})) - \eta (F(\bbz_k) - F(\bbz_{k-1}))
\label{eq:OGDA_eq_1}
\end{align}
On squaring Equation \eqref{eq:OGDA_eq_1} and using Young's inequality, we get:
\begin{align}
\eta^2 \| F(\bbz_{k+1}) - F(\bbz^*) \|^2 \leq 3\| \bbz_{k+1} - \bbz_{k}\|^2 + 3\eta^2 L^2\| \bbz_{k+1} - \bbz_{k} \|^2 + 3\eta^2 L^2\| \bbz_{k} - \bbz_{k-1}\|^2
\end{align}
Now, using strong convexity, and substituting $\eta = 1/4L$, we get:
\begin{align}
\label{eq:sub_back_OGDA}
\frac{\mu^2}{16 L^2} \| \bbz_{k+1} - \bbz^* \|^2 \leq 4 \max \{ \| \bbz_{k+1} - \bbz_{k}\|^2, \| \bbz_{k} - \bbz_{k-1}\|^2  \}
\end{align}
The following part of the proof is inspired by the result of Theorem $1$ of \cite{gidel2018variational}. For OGDA iterates, we have:
\begin{align}
\bbz_k - \eta(F(\bbz_k) - F(\bbz_{k-1})) = \bbz_{k-1} - \eta(F(\bbz_{k-1}) - F(\bbz_{k-2})) - \eta F(\bbz_k)
\label{eq:OGDA_mod_update}
\end{align}
On subtracting $\bbz^*$ from both sides and squaring, we have:
\begin{align}
\| \bbz_k - \eta(F(\bbz_k) & - F(\bbz_{k-1})) - \bbz^* \|^2 \nonumber \\
&= \| \bbz_{k-1} - \eta(F(\bbz_{k-1}) - F(\bbz_{k-2})) - \bbz^* \|^2 + \eta^2 \|F(\bbz_k)\|^2 \nonumber \\
& \qquad \qquad \qquad - 2\eta \langle F(\bbz_k), \bbz_{k-1} - \eta(F(\bbz_{k-1}) - F(\bbz_{k-2})) - \bbz^*\rangle \nonumber \\
&=  \| \bbz_{k-1} - \eta(F(\bbz_{k-1}) - F(\bbz_{k-2})) - \bbz^* \|^2 - 2\eta  \langle F(\bbz_k), \bbz_k - \bbz^* \rangle \nonumber \\
& \qquad \qquad \qquad -2 \langle \eta F(\bbz_k), \eta F(\bbz_{k-1}) \rangle + \eta^2 \|F(\bbz_k)\|^2 \nonumber \\
&= \| \bbz_{k-1} - \eta(F(\bbz_{k-1}) - F(\bbz_{k-2})) - \bbz^* \|^2 - 2\eta  \langle F(\bbz_k), \bbz_k - \bbz^* \rangle \nonumber \\
& \qquad \qquad \qquad +\eta^2 \| F(\bbz_k) - F(\bbz_{k-1}) \|^2 - \eta^2 \|F(\bbz_{k-1})\|^2 \nonumber \\
&\leq \| \bbz_{k-1} - \eta(F(\bbz_{k-1}) - F(\bbz_{k-2})) - \bbz^* \|^2 - 2\eta  \langle F(\bbz_k), \bbz_k - \bbz^* \rangle \nonumber \\
& \qquad \qquad \qquad +\eta^2 L^2 \| \bbz_k -\bbz_{k-1} \|^2 - \eta^2 \|F(\bbz_{k-1})\|^2
\label{ineq:1}
\end{align}
However, since:
\begin{align}
\langle F(\bbz_k) , \bbz_k - \bbz^* \rangle \geq \mu \|  \bbz_k - \bbz^* \|^2
\label{ineq:2}
\end{align}
and using Young's inequality we have
\begin{align}
\label{ineq:3}
\|  \bbz_k - \bbz^* \|^2 \leq \frac{1}{2} \| \bbz_{k-1} - \eta(F(\bbz_{k-1}) - F(\bbz_{k-2})) - \bbz^* \|^2 - \| \eta F(\bbz_{k-1}) \|^2
\end{align}
Substituting Equations \eqref{ineq:2} and \eqref{ineq:3} in Equation \eqref{ineq:1}, we have:
\begin{align}
\eta \mu & \left(\| \bbz_{k-1} - \eta(F(\bbz_{k-1}) - F(\bbz_{k-2})) - \bbz^* \|^2 - 2\eta^2 \| F(\bbz_{k-1}) \|^2  \right) \nonumber \\
& \leq \| \bbz_{k-1} - \eta(F(\bbz_{k-1}) - F(\bbz_{k-2})) - \bbz^* \|^2 - \| \bbz_k - \eta(F(\bbz_k) - F(\bbz_{k-1})) - \bbz^* \|^2 \nonumber \\
& \qquad \qquad +\eta^2 L^2 \| \bbz_k -\bbz_{k-1} \|^2 - \eta^2 \|F(\bbz_{k-1})\|^2
\end{align}
which on rearranging gives:
\begin{align}
\| \bbz_k - \eta(F(\bbz_k) & - F(\bbz_{k-1})) - \bbz^* \|^2 \nonumber \\
& \leq (1 - \eta \mu) \| \bbz_{k-1} - \eta(F(\bbz_{k-1}) - F(\bbz_{k-2})) - \bbz^* \|^2 + \eta^2 L^2 \| \bbz_k -\bbz_{k-1} \|^2 \nonumber \\
&  \qquad \qquad \qquad - \eta^2 (1 - 2 \eta \mu) \|F(\bbz_{k-1})\|^2
\label{eq:OGDA_sub_1}
\end{align}
However, for the OGDA iterates:
\begin{align}
\| \bbz_k - \bbz_{k-1} \|^2 &= \eta^2 \| F(\bbz_{k-1} + F(\bbz_{k-1}) - F(\bbz_{k-2}) \|^2 \nonumber \\
& \leq 2\eta^2 \|F(\bbz_{k-1} \|^2 + 2\eta^2 \| F(\bbz_{k-1}) - F(\bbz_{k-2}) \|^2 \nonumber \\
& \leq 2\eta^2 \|F(\bbz_{k-1}) \|^2 + 2\eta^2 L^2 \| \bbz_{k-1} - \bbz_{k-2} \|^2
\end{align}
which can be written as:
\begin{align}
\| \bbz_k - \bbz_{k-1} \|^2 \leq 4\eta^2 \|F(\bbz_{k-1}) \|^2 + 4\eta^2 L^2 \| \bbz_{k-1} - \bbz_{k-2} \|^2 - \| \bbz_k - \bbz_{k-1} \|^2 
\label{eq:OGDA_sub_2}
\end{align}
Substituting Equation \eqref{eq:OGDA_sub_2} in Equation \eqref{eq:OGDA_sub_1}, we get:
\begin{align}
\| \bbz_k & - \eta(F(\bbz_k)  - F(\bbz_{k-1})) - \bbz^* \|^2 + \eta^2 L^2 \| \bbz_k - \bbz_{k-1} \|^2 \nonumber \\
& \leq (1 - \eta \mu) \| \bbz_{k-1} - \eta(F(\bbz_{k-1}) - F(\bbz_{k-2})) - \bbz^* \|^2 + 4 \eta^4L^4  \| \bbz_{k-1} - \bbz_{k-2} \|^2 \nonumber \\
& \qquad \qquad -\eta^2 (1 - 2\eta \mu - 4 \eta^2 L^2) \| F (\bbz_{k-1})\|^2
\end{align}
For $\eta \leq 1/4L$, we have: $1 - 2\eta \mu - 4 \eta^2 L^2 > 0$ and therefore, can ignore the last term, which gives us (for  $\eta \leq 1/4L$)
\begin{align}
\| \bbz_k & - \eta(F(\bbz_k)  - F(\bbz_{k-1})) - \bbz^* \|^2 + \eta^2 L^2 \| \bbz_k - \bbz_{k-1} \|^2 \nonumber \\
& \leq (1 - \eta \mu) \| \bbz_{k-1} - \eta(F(\bbz_{k-1}) - F(\bbz_{k-2})) - \bbz^* \|^2 + 4 \eta^4L^4  \| \bbz_{k-1} - \bbz_{k-2} \|^2
\end{align}
since $\eta \leq 1/4L$, we have $(1- \eta \mu) \geq 4 \eta^2 L^2$, which gives:
\begin{align}
\| \bbz_k & - \eta(F(\bbz_k)  - F(\bbz_{k-1})) - \bbz^* \|^2 + \eta^2 L^2 \| \bbz_k - \bbz_{k-1} \|^2 \nonumber \\
& \leq (1 - \eta \mu) \left( \| \bbz_{k-1} - \eta(F(\bbz_{k-1}) - F(\bbz_{k-2})) - \bbz^* \|^2 + \eta^2 L^2 \| \bbz_{k-1} - \bbz_{k-2} \|^2 \right)
\end{align}
which gives us:
\begin{align}
\| \bbz_k & - \eta(F(\bbz_k)  - F(\bbz_{k-1})) - \bbz^* \|^2 + \eta^2 L^2 \| \bbz_k - \bbz_{k-1} \|^2 \leq (1 - \eta \mu)^k \left( \| \bbz_{0} - \bbz^* \|^2  \right)
\end{align}
in particular, for $\eta = \frac{1}{4L}$:
\begin{align}
 \| \bbz_k - \bbz_{k-1} \|^2 \leq 16 \left( 1 - \frac{\mu}{4L} \right) ^k \left( \| \bbz_{0} - \bbz^* \|^2  \right)
 \label{eq:final_ineq_OGDA_scsc}
\end{align}
Substituting Equation \eqref{eq:final_ineq_OGDA_scsc} back in Equation \eqref{eq:sub_back_OGDA}, we get:
\begin{align}
\| \bbz_{k+1} - \bbz^* \|^2 \leq \left( 1 - \frac{\mu}{4L} \right) ^k \times (1024 \kappa^2 \| \bbz_{0} - \bbz^* \|^2 )
\end{align}
On defining $\hat{r}_0 = 1024 \kappa^2 \| \bbz_{0} - \bbz^* \|^2$, we have:
\begin{align}
\| \bbz_{k+1} - \bbz^* \|^2 \leq \left( 1 - \frac{\mu}{4L} \right) ^k \hat{r}_0 .
\end{align}

\subsection{Proof of Theorem \ref{thm:bili_g}}

The generalized OGDA method for bilinear problems is given by:
\begin{align}
\bbx_{k+1} &= \bbx_k - (\alpha +\beta) \bbB \bby_k + \beta \bbB \bby_{k-1} \nonumber \\
\bby_{k+1} &= \bby_k + (\alpha +\beta) \bbB \bby_k - \beta \bbB \bby_{k-1} \nonumber 
\end{align}

We compare this with the Proximal Point (PP) method with stepsize $\alpha$. The proof follows long the exact same lines as the proof of Theorem \ref{thm:bili}.

We define the following symmetric matrices
\begin{align}
\bbE_x &= \bbI - \alpha^2  \bbB \bbB^{\top} - (\bbI + \alpha^2 \bbB \bbB^{\top})^{-1}, \nonumber \\
\bbE_y &= \bbI - \alpha^2  \bbB^{\top}  \bbB - (\bbI + \alpha^2 \bbB^{\top}  \bbB)^{-1}.\nonumber
\end{align}
We rewrite the properties of $\bbE_x$ and  $\bbE_y$ which are
\begin{align}
\|\bbE_x\|, \|\bbE_y\| &\leq  \frac{\alpha^4 \lambda_{\max}(\bbB^{\top} \bbB)^2}{1 - \alpha^2 \sqrt{\lambda^2_{\max}(\bbB^{\top} \bbB)}} =e 
\end{align}
\begin{align}
\bbE_x \bbB &= \bbB \bbE_y  \\
\bbE_y \bbB^{\top} &= \bbB^{\top} \bbE_x 
\end{align}

Therefore, the error between the OGDA and Proximal updates for the variable $\bbx$ is given by
\begin{align}\label{jafang_g}
 (-\beta \bbB \bby_k + \beta \bbB \bby_{k-1} + \alpha^2 \bbB \bbB^{\top} \bbx_k - \alpha^3 \bbB \bbB^{\top} \bbB \bby_k) + \bbE_x(\bbx_k - \alpha \bbB \bby_k)
\end{align}
We first derive an upper bound for the term in the first parentheses $(-\beta \bbB \bby_k + \beta \bbB \bby_{k-1} + \alpha^2 \bbB \bbB^{\top} \bbx_k - \alpha^3 \bbB \bbB^{\top} \bbB \bby_k) $.

Using the generalized OGDA update, we have:
\begin{align}
-\beta \bbB \bby_k + \beta \bbB \bby_{k-1} &= -\beta \bbB (\bby_k - \bby_{k-1} ) \nonumber \\
&= -\beta \bbB ( (\alpha + \beta) \bbB^{\top} \bbx_{k-1} - \beta \bbB^{\top} \bbx_{k-2} )
\end{align}
Therefore, we can write
\begin{align}
 (-\beta \bbB \bby_k & + \beta \bbB \bby_{k-1}  + \alpha^2 \bbB \bbB^{\top} \bbx_k - \alpha^3 \bbB \bbB^{\top} \bbB \bby_k) \nonumber \\
&= (\alpha^2\bbB \bbB^{\top} \bbx_k  - \alpha \beta \bbB \bbB^{\top} \bbx_{k-1}  - \beta^2 \bbB \bbB^{\top} \bbx_{k-1} + \beta^2 \bbB \bbB^{\top} \bbx_{k-2} - \alpha^3 \bbB \bbB^{\top} \bbB \bby_k ) \nonumber 
\end{align}
Once again, using the generalized OGDA updates for $(\bbx_k - \bbx_{k-1})$ and $(\bbx_{k-1} - \bbx_{k-2})$, we have
\begin{align}\label{jafang22_g}
(\alpha^2\bbB \bbB^{\top} \bbx_k   &- \alpha \beta \bbB \bbB^{\top} \bbx_{k-1}   - \beta^2 \bbB \bbB^{\top} \bbx_{k-1} + \beta^2 \bbB \bbB^{\top} \bbx_{k-2} - \alpha^3 \bbB \bbB^{\top} \bbB \bby_k ) \nonumber \\
&= \alpha (\alpha - \beta) \bbB \bbB^{\top} \bbx_k - \alpha \beta (\alpha + \beta)  \bbB \bbB^{\top}  \bbB \bby_{k-1}  + \alpha \beta^2 \bbB \bbB^{\top}  \bbB \bby_{k-2}  \nonumber \\
& \qquad + \beta^2 (\alpha + \beta) \bbB \bbB^{\top}  \bbB \bby_{k-2} - \beta^3 \bbB \bbB^{\top}  \bbB \bby_{k-3} - \alpha^3 \bbB \bbB^{\top} \bbB \bby_k \nonumber \\
&= \alpha (\alpha - \beta) \bbB \bbB^{\top} \bbx_k - \alpha \beta (\alpha + \beta)  \bbB \bbB^{\top}  \bbB \bby_{k-1}  + \beta^2 (2\alpha + \beta) \bbB \bbB^{\top}  \bbB \bby_{k-2}  \nonumber \\
& \qquad  - \beta^3 \bbB \bbB^{\top}  \bbB \bby_{k-3} - \alpha^3 \bbB \bbB^{\top} \bbB \bby_k 
\end{align}
Therefore, considering the expressions in \eqref{jafang_g} and \eqref{jafang22_g} the error between the updates of OGDA and PP for the variable $\bbx$ can be written as
\begin{align}
&\big( \bbx_k - (\alpha + \beta) \bbB \bby_k + \beta \bbB \bby_{k-1} \big)  -  \big( (\bbI + \alpha^2 \bbB \bbB^{\top})^{-1} (\bbx_k - \alpha \bbB y_k) \big) \nonumber \\
&=  \bbE_x(\bbx_k - \alpha \bbB \bby_k)  + \alpha (\alpha - \beta) \bbB \bbB^{\top} \bbx_k - \alpha \beta (\alpha + \beta)  \bbB \bbB^{\top}  \bbB \bby_{k-1}   \nonumber \\
& \qquad  + \beta^2 (2\alpha + \beta) \bbB \bbB^{\top}  \bbB \bby_{k-2}  - \beta^3 \bbB \bbB^{\top}  \bbB \bby_{k-3} - \alpha^3 \bbB \bbB^{\top} \bbB \bby_k 
\end{align}

Now, the convergence proof follows along the same lines as the proof of Theorem \ref{thm:bili}. We set $\eta = \max \{ \alpha , \beta \}$, and we need the additional assumption:
\begin{align}
|\alpha - \beta | \leq \mathcal{O} (\eta^3 / \alpha) \label{cond_beta}
\end{align}
due to the presence of the term $ \alpha (\alpha - \beta) \bbB \bbB^{\top} \bbx_k$, Let 
\begin{align}
\alpha - K \alpha^2 \leq \beta \leq \alpha
\end{align}
On making these substitutions, we get the same result as Theorem \ref{thm:bili}.

\subsection{Proof of Proposition~\ref{prop:EG_Approx}}

The Extragradient updates can be written as
\begin{align}
\bbx_{k+1} &= \bbx_k - \eta \nabla_{\bbx} f(\bbx_k - \eta \nabla_{\bbx} f(\bbx_k, \bby_k), \bby_k + \eta \nabla_{\bby} f(\bbx_k, \bby_k)) \nonumber \\
\bby_{k+1} &= \bby_k +  \eta \nabla_{\bby} f(\bbx_k - \eta \nabla_{\bbx} f(\bbx_k, \bby_k),\bby_k + \eta \nabla_{\bby} f(\bbx_k, \bby_k)) \nonumber
\end{align}
By writing the Taylor's expansion of $\nabla_{\bbx} f$ we obtain that 
\begin{align}
& \nabla_{\bbx} f \left(\bbx_k - \eta \nabla_{\bbx} f(\bbx_k, \bby_k), \bby_k + \eta \nabla_{\bby} f(\bbx_k, \bby_k)\right)\nonumber\\
& =  \nabla_{\bbx} f(\bbx_k ,\bby_k) 
+   \nabla_{\bbx\bbx} f(\bbx_k ,\bby_k) [\bbx_k - \eta \nabla_{\bbx} f(\bbx_k, \bby_k)-\bbx_k] \nonumber \\
& \qquad \qquad  \qquad  \qquad  \qquad  \qquad  \qquad  \qquad  +   \nabla_{\bbx\bby} f(\bbx_k ,\bby_k) [\bby_k + \eta \nabla_{\bby} f(\bbx_k, \bby_k)-\bby_k] +o(\eta)\nonumber\\
& =  \nabla_{\bbx} f(\bbx_k ,\bby_k) 
- \eta  \nabla_{\bbx\bbx} f(\bbx_k ,\bby_k) \nabla_{\bbx} f(\bbx_k, \bby_k)
+ \eta  \nabla_{\bbx\bby} f(\bbx_k ,\bby_k) \nabla_{\bby} f(\bbx_k, \bby_k)+o(\eta).
\end{align}
Use this expression to write 
\begin{align}\label{eq:EG__primal_Approx}
\bbx_{k+1} 
&= \bbx_k - \eta \nabla_{\bbx} f(\bbx_k,\bby_k) + \eta^2 \nabla_{\bbx \bbx} f(\bbx_k,\bby_k) \nabla_{\bbx} f(\bbx_k, \bby_k)    \nonumber\\
&\quad   - \eta^2 \nabla_{\bbx \bby}f(\bbx_k,\bby_k) \nabla_{\bby} f(\bbx_{k},\bby_{k}) + o(\eta^2).
\end{align}
By following the same argument for $\bby$ we obtain
\begin{align}
\bby_{k+1} &= \bby_k + \eta \nabla_{\bbx} f(\bbx_k,\bby_k) + \eta^2 \nabla_{\bby \bby} f(\bbx_k,\bby_k) \nabla_{\bby} f(\bbx_{k},\bby_{k})  \nonumber\\
&\quad - \eta^2 \nabla_{\bby \bbx}f(\bbx_k,\bby_k) \nabla_{\bbx} f(\bbx_{k},\bby_{k}) + o(\eta^2) \label{eq:EG__Dual_Approx}
\end{align}

Now we find a second order approximation for the Proximal Point Method. 
Note that the update of the proximal point method for variable $\bbx$ can be written as
\begin{align}
\bbx_{k+1}  &= \bbx_k -\eta  \nabla_{\bbx} f (\bbx_{k+1},\bby_{k+1}) \nonumber\\
&= \bbx_k - \eta \nabla_{\bbx} f \big( \bbx_k - \eta \nabla_{\bbx} f(\bbx_{k+1}, \bby_{k+1}), \bby_k + \eta \nabla_{\bby} f(\bbx_{k+1},\bby_{k+1}) \big) 
\end{align}
where in the second equality we replaced $\bbx_{k+1}$ and $\bby_{k+1}$ in the gradient with their updates. Hence, using Taylor's series we can show that
\begin{align}\label{hassan_1}
\bbx_{k+1}  &= \bbx_k - \eta \nabla_{\bbx} f(\bbx_k,\bby_k) + \eta^2 \nabla_{\bbx \bbx} f(\bbx_k,\bby_k) \nabla_{\bbx} f(\bbx_{k+1},\bby_{k+1}) \nonumber \\
& \quad  - \eta^2 \nabla_{\bbx \bby}f(\bbx_k,\bby_k) \nabla_{\bby} f(\bbx_{k+1},\bby_{k+1}) + o(\eta^2) \nonumber \\
&= \bbx_k - \eta \nabla_{\bbx} f(\bbx_k,\bby_k) + \eta^2 \nabla_{\bbx \bbx} f(\bbx_k,\bby_k) \nabla_{\bbx} f(\bbx_{k},\bby_{k})   \nonumber\\
&\quad   - \eta^2 \nabla_{\bbx \bby}f(\bbx_k,\bby_k) \nabla_{\bby} f(\bbx_{k},\bby_{k}) + o(\eta^2) ,
\end{align}
where in the second equality we used the fact that $\nabla_{\bbx} f(\bbx_{k+1}, \bby_{k+1}) = \nabla_{\bbx} f(\bbx_k, \bby_k) + \mathcal{O}(\eta)$ and $\nabla_{\bby} f(\bbx_{k+1}, \bby_{k+1}) = \nabla_{\bby} f(\bbx_k, \bby_k) + \mathcal{O}(\eta)$. Similarly, we find the approximation of the update of $\bby$ which leads to
\begin{align}\label{hassan_2}
\bby_{k+1} &= \bby_k + \eta \nabla_{\bbx} f(\bbx_k,\bby_k) + \eta^2 \nabla_{\bby \bby} f(\bbx_k,\bby_k) \nabla_{\bby} f(\bbx_{k},\bby_{k})   \nonumber\\
&\quad - \eta^2 \nabla_{\bbx \bby}f(\bbx_k,\bby_k) \nabla_{\bbx} f(\bbx_{k},\bby_{k}) + o(\eta^2) .
\end{align}
Comparing the expressions in \eqref{eq:EG__primal_Approx} and \eqref{eq:EG__Dual_Approx} with the ones in \eqref{hassan_1} and \eqref{hassan_2} implies that the difference between the updates of PP and EG is at most $o(\eta^2)$ and this completes the proof.

\subsection{Proof of Theorem~\ref{thm:EG_bilinear}}
Define the following symmetric error matrices
\begin{align}\label{def_errors}
\bbE_x = \bbI - \eta^2  \bbB \bbB^{\top} - (\bbI + \eta^2 \bbB \bbB^{\top})^{-1} ,\qquad 
\bbE_y = \bbI - \eta^2  \bbB^{\top}  \bbB - (\bbI + \eta^2 \bbB^{\top}  \bbB)^{-1} 
\end{align}
which are useful to characterize the difference between the updates of EG and PP for a bilinear problem. Note that we can bound the norms of $\bbE_x$ and $\bbE_y$ as
\begin{align}
\|\bbE_x\| &\leq \eta^4 \sqrt{\lambda_{\max}^4(\bbB \bbB^{\top})} + \eta^6 \sqrt{ \lambda_{\max}^6(\bbB \bbB^{\top})} + \cdots \nonumber \\
&= \frac{\eta^4 \lambda_{\max}(\bbB \bbB^{\top})^2}{1 - \eta^2 \sqrt{\lambda^2_{\max}(\bbB \bbB^{\top})}},
\end{align}
and similarly 
\begin{align}
\|\bbE_y\| &\leq  \frac{\eta^4 \lambda_{\max}(\bbB^{\top} \bbB)^2}{1 - \eta^2 \sqrt{\lambda^2_{\max}(\bbB^{\top} \bbB)}}.
\end{align}
Since $\lambda_{\max}(\bbB^{\top} \bbB) =  \lambda_{\max}(\bbB \bbB^{\top})$, we have:
\begin{align}\label{err_def_1}
\|\bbE_x\|, \|\bbE_y\| &\leq  \frac{\eta^4 \lambda_{\max}(\bbB^{\top} \bbB)^2}{1 - \eta^2 \sqrt{\lambda^2_{\max}(\bbB^{\top} \bbB)}} := e
\end{align}
Also, from Lemma \ref{lemma:Matrix_follow} in the proof of Theorem \ref{thm:bilinear_ppm}, and the definitions of the error matrices in \eqref{def_errors} it can be verified that 
\begin{align}
\bbE_x \bbB &= \bbB \bbE_y \label{eq:error_flow_1} \\
\bbE_y \bbB^{\top} &= \bbB^{\top} \bbE_x \label{eq:error_flow_2} 
\end{align}
Moreover, using the definitions of $\bbE_x$ and $\bbE_y$ in \eqref{def_errors},  the EG updates can be written as
\begin{align}
\bbx_{k+1} &= (\bbI + \eta^2 \bbB \bbB^{\top})^{-1} (\bbx_k - \eta \bbB \bby_k) - \eta^3 \bbB \bbB^{\top} \bbB \bby_k + \bbE_x (\bbx_k - \eta \bbB \bby_k) \label{eq:EG_prox_P},\\
\bby_{k+1} &= (\bbI + \eta^2 \bbB^{\top} \bbB)^{-1} (\bby_k + \eta \bbB^{\top} \bbx_k) + \eta^3 \bbB^{\top} \bbB \bbB^{\top} \bbx_k + \bbE_y (\bby_k + \eta \bbB^{\top} \bbx_k) \label{eq:EG_prox_D}.
\end{align}
As in the proof of Theorem \ref{thm:bilinear_ppm}, we define $\bbQ_x = (\bbI + \eta^2 \bbB \bbB^{\top})^{-1}$ and $\bbQ_y = (\bbI + \eta^2 \bbB^{\top} \bbB)^{-1}$. Using these definitions we can show that
\begin{align}
\|\bbx_{k+1}\|^2 &= (\bbx_k - \eta \bbB \bby_k) ^{\top} \bbQ_x^2 (\bbx_k - \eta \bbB \bby_k) + \eta^6 \bby_k^{\top}  \bbB^{\top} \bbB  \bbB^{\top} \bbB  \bbB^{\top} \bbB \bby_k  \nonumber \\
& \qquad + (\bbx_k - \eta \bbB \bby_k) ^{\top} \bbE_x^2 (\bbx_k - \eta \bbB \bby_k) + 2(\bbx_k - \eta \bbB \bby_k) ^{\top} \bbE_x  \bbQ_x (\bbx_k - \eta \bbB \bby_k)  \nonumber \\
& \qquad  -2\eta^3 (\bbx_k - \eta \bbB \bby_k) ^{\top} \bbE_x  \bbB \bbB^{\top} \bbB \bby_k -2 \eta^3 \bby_k^{\top}  \bbB^{\top} \bbB \bbB^{\top}  \bbQ_x (\bbx_k - \eta \bbB \bby_k)\label{chert_1} \\
\|\bby_{k+1}\|^2 
&= (\bby_k + \eta \bbB^{\top} \bbx_k) ^{\top} \bbQ_y^2 (\bby_k + \eta \bbB^{\top} \bbx_k) + \eta^6 \bbx_k^{\top} \bbB \bbB^{\top} \bbB  \bbB^{\top} \bbB  \bbB^{\top}  \bbx_k + (\bby_k  \nonumber \\
& \qquad  + \eta \bbB^{\top} \bbx_k) ^{\top} \bbE_y^2 (\bby_k + \eta \bbB^{\top} \bbx_k) + 2(\bby_k + \eta \bbB^{\top} \bbx_k) ^{\top} \bbE_y  \bbQ_y (\bby_k + \eta \bbB^{\top} \bbx_k)  \nonumber \\
& \qquad + 2\eta^3 (\bby_k + \eta \bbB^{\top} \bbx_k) ^{\top} \bbE_y  \bbB^{\top} \bbB \bbB^{\top} \bbx_k +2 \eta^3 \bbx_k^{\top}  \bbB \bbB^{\top} \bbB  \bbQ_y (\bby_k + \eta \bbB^{\top} \bbx_k) \label{chert_2}
\end{align}
Now before adding the two sides of the expressions in \eqref{chert_1} and \eqref{chert_2}, note that some of the cross terms in \eqref{chert_1} and \eqref{chert_2} cancel out. For instance, using Lemma  \ref{lemma:Matrix_follow} and Equations \eqref{eq:error_flow_1} and \eqref{eq:error_flow_2} we can show that 
\begin{align}
- \eta^3 \bbx_k^{\top} \bbE_x \bbB \bbB^{\top} \bbB \bby_k + \eta^3 \bbx_k^{\top} \bbB \bbB^{\top}  \bbB \bbE_y \bby_k &=  - \eta^3 \bbx_k^{\top} \bbB \bbE_y \bbB^{\top} \bbB \bby_k + \eta^3 \bbx_k^{\top} \bbB \bbB^{\top}  \bbB \bbE_y \bby_k \nonumber \\
&=  - \eta^3 \bbx_k^{\top} \bbB \bbB^{\top} \bbE_x  \bbB \bby_k + \eta^3 \bbx_k^{\top} \bbB \bbB^{\top}  \bbB \bbE_y \bby_k \nonumber \\
&= - \eta^3 \bbx_k^{\top} \bbB \bbB^{\top}  \bbB \bbE_y \bby_k + \eta^3 \bbx_k^{\top} \bbB \bbB^{\top}  \bbB \bbE_y \bby_k\nonumber \\
&= 0 \nonumber
\end{align}
By using similar arguments it can be shown that summing  two sides of the expressions in \eqref{chert_1} and \eqref{chert_2} leads to 
\begin{align}
&\|\bbx_{k+1}\|^2+\|\bby_{k+1}\|^2\nonumber\\
 &=
 \bbx_k^{\top} \bbQ_x^2 \bbx_k
+ \eta^2 \bby_k^{\top}\bbB^{\top}  \bbQ_x^2 \bbB \bby_k
 + \eta^6 \bby_k^{\top}  (\bbB^{\top} \bbB)^3 \bby_k
  + \bbx_k^{\top} \bbE_x^2 \bbx_k
  +  \eta^2\bby_k^{\top} \bbB^{\top} \bbE_x^2  \bbB \bby_k  
  \nonumber \\
& \quad 
 + 2\bbx_k^{\top} \bbE_x  \bbQ_x \bbx_k 
  + 2\eta^2 \bby_k^{\top}\bbB^{\top} \bbE_x  \bbQ_x \bbB \bby_k
+2\eta^4  \bby_k^{\top}   \bbB^{\top}   \bbB \bbE_y \bbB^{\top} \bbB \bby_k 
+2 \eta^4 \bby_k^{\top}  \bbB^{\top} \bbB \bbQ_y \bbB^{\top}   \bbB \bby_k\nonumber\\
&\quad + \bby_k^{\top} \bbQ_y^2 \bby_k 
+ \eta^2 \bbx_k^{\top}\bbB \bbQ_y^2 \bbB^{\top} \bbx_k
+ \eta^6 \bbx_k^{\top} (\bbB \bbB^{\top})^3 \bbx_k
 + \bby_k^{\top} \bbE_y^2 \bby_k
  + \eta^2  \bbx_k^{\top}\bbB \bbE_y^2  \bbB^{\top} \bbx_k 
 \nonumber \\
& \quad  
+ 2\bby_k ^{\top} \bbE_y\bbQ_y \bby_k 
+ 2\eta^2\bbx_k^{\top} \bbB \bbE_y  \bbQ_y \bbB^{\top} \bbx_k
+ 2\eta^4  \bbx_k ^{\top} \bbB\bbB^{\top}\bbE_x   \bbB \bbB^{\top} \bbx_k  +2 \eta^4 \bbx_k^{\top}  \bbB \bbB^{\top} \bbQ_x\bbB    \bbB^{\top} \bbx_k 
\nonumber\\
 &=
 \bbx_k^{\top} \bbQ_x \bbx_k
 + \eta^6 \bby_k^{\top}  (\bbB^{\top} \bbB)^3 \bby_k
  + \bbx_k^{\top} \bbE_x^2 \bbx_k
  +  \eta^2\bby_k^{\top} \bbB^{\top} \bbE_x^2  \bbB \bby_k  
  \nonumber \\
& \quad 
 + 2\bbx_k^{\top} \bbE_x \bbx_k 
+2\eta^4  \bby_k^{\top}   \bbB^{\top}   \bbB \bbE_y \bbB^{\top} \bbB \bby_k 
+2 \eta^4 \bby_k^{\top}  \bbB^{\top} \bbB \bbQ_y \bbB^{\top}   \bbB \bby_k\nonumber\\
&\quad + \bby_k^{\top} \bbQ_y \bby_k 
+ \eta^6 \bbx_k^{\top} (\bbB \bbB^{\top})^3 \bbx_k
 + \bby_k^{\top} \bbE_y^2 \bby_k
  + \eta^2  \bbx_k^{\top}\bbB \bbE_y^2  \bbB^{\top} \bbx_k 
 \nonumber \\
& \quad  
+ 2\bby_k ^{\top} \bbE_y \bby_k 
+ 2\eta^4  \bbx_k ^{\top} \bbB\bbB^{\top}\bbE_x   \bbB \bbB^{\top} \bbx_k  +2 \eta^4 \bbx_k^{\top}  \bbB \bbB^{\top} \bbQ_x\bbB    \bbB^{\top} \bbx_k 
\end{align}
where in the second equality we used the simplifications
\begin{align}
&\bbx_k^{\top} \bbQ_x^2 \bbx_k+ \eta^2 \bbx_k^{\top}\bbB \bbQ_y^2 \bbB^{\top} \bbx_k
= \bbx_k^{\top} \bbQ_x^2 \bbx_k+ \eta^2 \bbx_k^{\top}\bbQ_x^2\bbB \bbB^{\top}  \bbQ_x\bbx_k=  \bbx_k^{\top} \bbQ_x \bbx_k\nonumber\\
& \bby_k^{\top} \bbQ_y^2 \bby_k  + \eta^2 \bby_k^{\top}\bbB^{\top}  \bbQ_x^2 \bbB \bby_k
=  \bby_k^{\top} \bbQ_y^2 \bby_k + \eta^2 \bby_k^{\top}\bbQ_y^2\bbB^{\top} \bbB \bbQ_y\bby_k=  \bby_k^{\top} \bbQ_y \bby_k
\end{align}
as well as
\begin{align}
&\bbx_k^{\top} \bbE_x  \bbQ_x \bbx_k + \eta^2\bbx_k^{\top} \bbB \bbE_y  \bbQ_y \bbB^{\top} \bbx_k
= \bbx_k^{\top} \bbE_x  \bbQ_x \bbx_k + \eta^2\bbx_k^{\top}  \bbE_x\bbQ_x \bbB  \bbB^{\top} \bbx_k = \bbx_k^{\top} \bbE_x   \bbx_k
 \nonumber\\
&\bby_k ^{\top} \bbE_y\bbQ_y  \bby_k  +  \eta^2 \bby_k^{\top}\bbB^{\top} \bbE_x  \bbQ_x \bbB \bby_k =
\bby_k ^{\top} \bbE_y\bbQ_y  \bby_k  +  \eta^2 \bby_k^{\top}\bbE_y  \bbQ_y \bbB^{\top} \bbB \bby_k=\bby_k ^{\top} \bbE_y  \bby_k 
\end{align}
Define $r_k = \|\bbx_k\|^2 + \|\bby_k\|^2$. We have:
\begin{align}
& r_{k+1} \leq  \nonumber \\
& \bbx_k^{\top} (\bbQ_x + 2\bbE_x + \bbE_x^2 + \eta^6 (\bbB \bbB^{\top})^3 + \eta^2 \bbB \bbE_y^2 \bbB^{\top} + 2\eta^4 \bbB \bbB^{\top} \bbE_x \bbB \bbB^{\top} + 2\eta^4 \bbB \bbB^{\top} \bbQ_x \bbB \bbB^{\top} ) \bbx_k \nonumber \\
& + \bby_k^{\top} (\bbQ_y + 2\bbE_y + \bbE_y^2 + \eta^6 (\bbB^{\top} \bbB)^3 + \eta^2 \bbB^{\top} \bbE_x^2 \bbB + 2\eta^4 \bbB^{\top} \bbB \bbE_y \bbB^{\top} \bbB + 2\eta^4 \bbB^{\top} \bbB \bbQ_y \bbB^{\top} \bbB) \bby_k
\end{align}
Choosing $\eta = \frac{1}{2 \sqrt{2 \lambda_{\max} (\bbB^{\top} \bbB)}} $, we have:
\begin{align}
r_{k+1} &\leq \bbx^{\top}_k (\bbI - \frac{1}{2}\eta^2\bbB \bbB^{\top} + \frac{1}{4}\eta^2 \bbB \bbB^{\top}  ) \bbx_k + \bby_k^{\top} (\bbI - \frac{1}{2}\eta^2\bbB^{\top} \bbB + \frac{1}{4}\eta^2 \bbB^{\top} \bbB) \bby_k \nonumber \\
&\leq \left( 1 - \frac{1}{20 \kappa} \right) r_k
\end{align}

\subsection{Proof of Theorem~\ref{thm:EG_scv}}

Define $\bbz = [\bbx; \bby]$ and $F(\bbz) = [\nabla_{\bbx} f(\bbx, \bby) ; -\nabla_{\bby} f(\bbx, \bby)]$. Then the EG updates can be compactly written as:
\begin{align}
\bbz_{k+1} = \bbz_k - \eta F(\bbz_{k+1/2})
\end{align}
where
\begin{align}
\bbz_{k+1/2} = \bbz_k - \eta F(\bbz_{k})
\end{align}
We write the update in terms of the Proximal Point method with an error $\bbvarepsilon_k = \eta(F(\bbz_{k+1}) - F(\bbz_{k+1/2}))$ as follows:
\begin{align}
\bbz_{k+1} = \bbz_k - \eta F(\bbz_{k+1}) + \bbvarepsilon_k
\end{align}
On squaring and simplifying this expression, we have:
\begin{align}
\| \bbz_{k+1} - \bbz^* \|^2 = \| \bbz_k - \bbz^*\|^2 - \| \bbz_{k+1} - \bbz_{k} \|^2 - 2\eta (F(\bbz_{k+1}) + \bbvarepsilon_k)^{\top} (\bbz_{k+1} - \bbz^*)
\label{eq:EG_Main_1}
\end{align}
where $\bbz^* = [ \bbx^*; \bby^*]$ (Note that $r_k = \|\bbz_k - \bbz^*\|^2$). We simplify the right hand side of Equation \eqref{eq:EG_Main_1} as follows-
\begin{align}
& \| \bbz_k - \bbz^*\|^2 - \| \bbz_{k+1} - \bbz_{k} \|^2 - 2\eta (F(\bbz_{k+1}) + \bbvarepsilon_k)^{\top} (\bbz_{k+1} - \bbz^*) \nonumber \\
&= \| \bbz_k - \bbz^*\|^2 - \| \bbz_{k+1} - \bbz_{k} \|^2 - 2\eta (F(\bbz_{k+1/2}))^{\top} (\bbz_{k+1} - \bbz^*) \nonumber \\
&= \| \bbz_k - \bbz^*\|^2  - 2\eta (F(\bbz_{k+1/2}))^{\top} (\bbz_{k+1} - \bbz^*) - \| \bbz_{k+1} - \bbz_{k+1/2} + \bbz_{k+1/2} -\bbz_{k} \|^2 \nonumber \\
&=  \| \bbz_k - \bbz^*\|^2  - 2\eta (F(\bbz_{k+1/2}))^{\top} (\bbz_{k+1} - \bbz^*) - \| \bbz_{k+1} - \bbz_{k+1/2}\|^2 - \|\bbz_{k+1/2} -\bbz_{k} \|^2 \nonumber \\
& \qquad \qquad - 2(\bbz_{k+1} - \bbz_{k+1/2})^{\top} (\bbz_{k+1/2} -\bbz_{k}) \nonumber \\
&= \| \bbz_k - \bbz^*\|^2  - 2\eta (F(\bbz_{k+1/2}))^{\top} (\bbz_{k+1} - \bbz^*) - \| \bbz_{k+1} - \bbz_{k+1/2}\|^2 - \|\bbz_{k+1/2} -\bbz_{k} \|^2 \nonumber \\
& \qquad \qquad - 2\eta (F(\bbz_{k}))^{\top} (\bbz_{k+1/2} -\bbz_{k+1}) 
\label{eq:EG_Main_2}
\end{align}
The following part of the proof is inspired by the result of Theorem $1$ of \cite{gidel2018variational}. We simplify the inner products and give a lower bound using strong convexity as follows:
\begin{align}
2 & \eta (F(\bbz_{k+1/2}))^{\top} (\bbz_{k+1} - \bbz^*) + 2\eta (F(\bbz_{k}))^{\top} (\bbz_{k+1/2} -\bbz_{k+1}) \nonumber \\
&= 2\eta (F(\bbz_{k+1/2}))^{\top} (\bbz_{k+1/2} - \bbz^*) + 2\eta (F(\bbz_{k}) - F(\bbz_{k+1/2}))^{\top} (\bbz_{k+1/2} -\bbz_{k+1}) \nonumber \\
&\geq 2\eta \mu \| \bbz_{k+1/2} - \bbz^* \| - 2\eta L \| \bbz_{k} - \bbz_{k+1/2} \| \| \bbz_{k+1/2} - \bbz_{k+1} \| 
\end{align}
since $F(\bbz^*) = 0$. Now, using Young's inequality, we have:
\begin{align}
2 & \eta (F(\bbz_{k+1/2}))^{\top} (\bbz_{k+1} - \bbz^*) + 2\eta (F(\bbz_{k}))^{\top} (\bbz_{k+1/2} -\bbz_{k+1}) \nonumber \\
&\geq 2\eta \mu \| \bbz_{k+1/2} - \bbz^* \|^2 - (\eta^2L^2 \| \bbz_{k} - \bbz_{k+1/2} \| + \| \bbz_{k+1/2} - \bbz_{k+1} \| ^2)
\end{align}
Substituting the above inequality in Equation \eqref{eq:EG_Main_2}, we have:
\begin{align}
\| \bbz_{k+1} - \bbz^*\|^2 \leq \|\bbz_k - \bbz^*\|^2 - 2\eta \mu \| \bbz_{k+1/2} - \bbz^* \|^2 + (\eta^2L^2 -1) \| \bbz_{k} - \bbz_{k+1/2} \| \end{align}
Since $\| \bbz_{k+1/2} - \bbz^*\|^2 \leq 2\| \bbz_k - \bbz^*\|^2 + 2\| \bbz_{k+1/2} - \bbz_k\|^2$, we have:
\begin{align}
\| \bbz_{k+1} - \bbz^*\|^2 \leq (1 - \eta \mu) \|\bbz_k - \bbz^*\|^2 + (\eta^2L^2 + 2\eta \mu -1) \| \bbz_{k} - \bbz_{k+1/2} \| 
\end{align}
For $\eta = 1/4L$, we have $\eta^2L^2 + 2\eta \mu -1 < 1$ (since $\mu \leq L$), which gives:
\begin{align}
\| \bbz_{k+1} - \bbz^*\|^2 \leq \left(1 - \frac{1}{4\kappa} \right) \|\bbz_k - \bbz^*\|^2
\end{align}
where $\kappa = \frac{L}{\mu}$.